\documentclass[11pt]{article}

\usepackage{amsmath,amssymb,amsthm}
\usepackage{amsfonts}
\usepackage[mathscr]{eucal}
\usepackage{color}

\begin{document}
\newtheorem{thm}{Theorem}
\numberwithin{thm}{section}
\newtheorem{lemma}[thm]{Lemma}
\newtheorem{remark}{Remark}
\newtheorem{example}[thm]{Example}
\newtheorem{corr}[thm]{Corollary}
\newtheorem{proposition}{Proposition}
\newtheorem{theorem}{Theorem}[section]
\newtheorem{deff}[thm]{Definition}
\newtheorem{case}[thm]{Case}
\newtheorem{prop}[thm]{Proposition}
\numberwithin{equation}{section}
\numberwithin{remark}{section}
\numberwithin{proposition}{section}
\newtheorem{corollary}{Corollary}[section]
\newtheorem{others}{Theorem}
\newtheorem{conjecture}{Conjecture}
\newtheorem{definition}{Definition}[section]
\newtheorem{cl}{Claim}
\newtheorem{cor}{Corollary}
\newcommand{\ds}{\displaystyle}
%\date{}
\newcommand{\alert}[1]{\fbox{#1}}
\newcommand{\stk}[2]{\stackrel{#1}{#2}}
\newcommand{\dwn}[1]{{\scriptstyle #1}\downarrow}
\newcommand{\upa}[1]{{\scriptstyle #1}\uparrow}
\newcommand{\nea}[1]{{\scriptstyle #1}\nearrow}
\newcommand{\sea}[1]{\searrow {\scriptstyle #1}}
\newcommand{\csti}[3]{(#1+1) (#2)^{1/ (#1+1)} (#1)^{- #1
 / (#1+1)} (#3)^{ #1 / (#1 +1)}}
\newcommand{\RR}[1]{\mathbb{#1}}
\def\Re {{\rm Re}\,}
\def\Im {{\rm Im}\,}
\def\E{{\mathbb E}}
\def\P{{\mathbb P}}
\def\R{{\mathbb R}}
\newcommand{\rd}{{\mathbb R^d}}
\def\Z{{\mathbb Z}}
\def\N{{\mathbb N}}
\def\ep{\varepsilon}
\def\l{{\langle}}
\def\r{\rangle}
\def\La{\Lambda}
\def\si{\sigma}
\def\ga{\gamma}
\def\Ga{\Gamma}
\def\la{\lambda}

\def\M{{\EuScript M}}
\def\EN{{\EuScript{E}}}

\thispagestyle{empty}
\begin{titlepage}
\title{\bf  $\alpha$-Time
Fractional Brownian Motion: PDE Connections and Local Times}
\author{Erkan Nane\\
Auburn University\\
\and Dongsheng Wu\\
University of Alabama in Huntsville\\
\and
Yimin Xiao
\thanks{Research partially
supported by NSF grant DMS-1006903.}\\
Michigan State University }
\maketitle

\begin{abstract}
\noindent {\it
For $0<\alpha \leq 2$ and $0<H<1$, an $\alpha$-time fractional
Brownian motion is an iterated process $Z = \{Z(t)=W(Y(t)), t \ge 0\}$
obtained by taking a fractional Brownian motion $\{W(t), t\in \RR{R} \}$
with Hurst index $0<H<1$ and replacing the time parameter with a
strictly $\alpha$-stable L\'evy process $\{Y(t), t\geq 0 \}$ in
$\RR{R}$ independent of $\{W(t), t \in \R\}$. It is shown that such
processes have natural connections to partial differential equations
and, when $Y$ is a stable subordinator, can arise as scaling limit
of randomly indexed random walks. The existence, joint continuity
and sharp H\"older conditions in the set variable of the local times
of a $d$-dimensional $\alpha$-time fractional Brownian motion $X =
\{X(t), t \in \R_+$\} defined by
$
X(t)=\big(X_{1}(t),\cdots , X_{d}(t) \big),
$
where $t\geq 0$  and $X_{1},\cdots , X_{d}$ are independent copies
of $Z$, are investigated. Our methods rely on the strong local '
nondeterminism of fractional Brownian motion.
}
\end{abstract}

\textbf{Key words:} Fractional Brownian motion, strictly
$\alpha$-stable L\'evy process, $\alpha$-time Brownian motion,
$\alpha$-time fractional Brownian motion, partial differential equation,
local time, H\"older condition.
\newline

\textbf{Mathematics Subject Classification (2000):} 60G17, 60J65,
60K99.
\end{titlepage}

\section{Introduction}

In recent years, iterated Brownian motion and related iterated
processes have received much research interest. Such iterated
processes are connected naturally with partial differential equations
and have interesting probabilistic and statistical features such
as self-similarity, non-Markovian dependence structure,
non-Gaussian distributions; see \cite{allouba1,burdzy1,burdzy-khos,
BK98,deblassie,dovidio-orsingher,nane1,nane2,OB1,xiao}
and references therein for further information. Inspired by these
results, we consider a new class of iterated processes called
$\alpha$-time fractional Brownian motion (fBm) for $0<\alpha \leq 2$
and $0<H<1$. These are obtained by taking a fractional Brownian motion of
index $H$ and replacing the time parameter with  a strictly
$\alpha$-stable L\'evy process $Y$. More precisely, let
$W = \{W(t), t \in \R\}$ be a fractional Brownian motion
in $\RR{R}$ with index $H$, which is a centered, real-valued
Gaussian process with covariance function
$$
\E\big(W (t) W(s) \big) = \frac1 2\big(|t|^{2H} +
|s|^{2 H} - |t - s|^{2 H}\big)
$$
and $W(0) = 0$ a.s. Here and in the sequel,
$|\cdot|$ denotes the Euclidean norm. Let $Y = \{Y(t),
t \ge 0\}$ be a real-valued strictly $\alpha$-stable L\'evy process,
$0<\alpha \leq 2$, starting from $0$; see Section 3 for its definition
and \cite{bertoin, sato} for further information. We assume that $W$
and $Y$ are independent. Then a real-valued $\alpha$-time fractional
Brownian motion $Z= \{Z(t), t \ge 0\}$ is defined by
\begin{equation}\label{definition}
Z(t)\equiv W(Y(t)), \ \ \ t\geq 0.
\end{equation}

For $\alpha =2$ and $H=1/2$, this is the iterated Brownian
motion of Burdzy \cite{burdzy1}. When $0<\alpha<2$ and
$H=1/2$, $Z$ is called an $\alpha$-time Brownian motion
by Nane \cite{nane-alpha}. Aurzada and Lifshits
\cite{aurzada-lifshits} and Linde and Shi \cite{linde-shi}
studied the small deviation problem for real-valued $\alpha$-time
Brownian motion. Nane \cite{nane-lil} studied laws of the iterated
logarithm for a version of $Z$. Moreover, when $Y$ is symmetric,
for $\alpha =1,\,2$ and $H=1/2$ these processes have
connections with partial differential
operators as described in \cite{allouba1,nane3}.

More generally, it is easy to verify that the process $Z$ has
stationary increments and is a self-similar process of index
$H/\alpha$. The latter means that, for every constant $c>0$, the
processes $\{Z(t):\ t\geq 0 \}$ and $\{c^{-H/\alpha}Z(c\,t):\ t\geq 0\}$
have the same finite-dimensional distributions. Gaussian and stable
self-similar processes have been studied extensively in recent years;
see Samorodnitsky and Taqqu \cite{ST94}, Embrechts and Maejima
\cite{EM02} for further information. The $\alpha$-time fractional
Brownian motions form an important class of non-Markovian and
non-stable self-similar processes, except in the special case
when $H = 1/2$ and $Y$ is a stable subordinator [In this case,
$Z$ is a symmetric stable L\'evy process]. As will be shown in
this paper, they have natural connections to partial differential
equations and can arise as scaling limit of randomly indexed random
walks with dependent jumps. Hence they can serve as useful stochastic
models in many scientific areas including physics, insurance risk
theory and communication networks. Moreover, because they are
non-Markovian and have non-stable distributions, new methods are
often needed in order to study their properties.

When $\alpha < 2$, the sample function of the $\alpha$-time
fractional Brownian motion $Z$ is not continuous and its irregularity
is closely related to those of $W$ and $Y$. One of our motivations
of this paper is to characterize the irregularity of $Z$
in terms of the parameters $H$ and $\alpha$. We do this by
studying the existence and regularity of the local times of
$\alpha$-time fractional Brownian motion
$X = \{X(t), t \ge 0\}$ with values in $\R^d$ defined by
\begin{equation}\label{multi-alpha}
X(t)=\big(X_{1}(t),\cdots , X_{d}(t) \big),\ \ \ \ \ \ (t\geq 0),
\end{equation}
where $X_{j}= W_j(Y_j(t))$ ($j=1, \cdots, d$). We assume that
$W_1, \cdots,W_d$ are independent copies of $W$, $Y_1,
\cdots, Y_d$ are independent copies of $Y$, and $\{W_j\}$
and $\{Y_j\}$ are independent. We will call $X= \{X(t),
t \ge 0\}$ a $d$-dimensional $\alpha$-time fractional Brownian
motion. It is clear that $X$ is also self-similar of index
$H/\alpha$ and has stationary increments.

The rest of this paper is organized as follows. In Section 2, we
study the PDE connections of $\alpha$-time
fractional Brownian motions, and prove that they can be
obtained as scaling limit of randomly indexed random walks
with dependent jumps. These results provide some analytic and
physical interpretations for $\alpha$-time fractional Brownian
motions.

In Sections 3 and 4 we investigate the existence, joint continuity
and sharp H\"{o}lder conditions in the set variable of the local
times of a $d$ dimensional $\alpha$-time fractional Brownian motion
$X$. In the special case of $d=1$, $W$ is Brownian motion and
$Y$ is a symmetric $\alpha$-stable L\'evy process with $\alpha > 1$,
the existence and joint continuity of the local time of $Z$ have
been proved by Nane \cite{nane-alpha}. The methods used in this
paper differ from those of Nane \cite{nane-alpha}. The latter uses
the existence of local times of Brownian motion in $\RR{R}$ as
well as the existence of local time of symmetric stable L\'evy
processes, which does not exist whenever $\alpha \leq 1$ or $d > 1$.
Our Theorem \ref{fourier} implies that, for the $\alpha$-time
Brownian motion $X$ in $\R^d$,  local times exist in
the case $d=1$ for $\alpha>1/2$;  in the case $d=2$ for
$\alpha >1$ and in the case $d=3$ for $\alpha>3/2$. Moreover,
Theorem \ref{main-thm-1} shows that these local times have
jointly continuous versions.

The methods of  Sections 3 and 4 rely on the Fourier analytic
argument of Berman \cite{Berman69, Berman73} and a chaining
argument in Ehm \cite{ehm}. In order to derive crucial moment
estimates for the local times, we make use of the strong local
nondeterminism (SLND) of fractional Brownian motion proved
by Pitt \cite{Pitt78} as well as several nontrivial
modifications of the arguments in Xiao \cite{xiao2, xiao}.

Finally, we provide some technique lemmas as an Appendix in Section 5.

Throughout the paper, we will use $K$ to denote an unspecified positive
finite constant which may not necessarily be the same in each occurrence.

%\section{Preliminaries}
\section{PDE connections and scaling limits of randomly
indexed random walks}

In this section we show that $\alpha$-time fractional
Brownian motions have natural connections to partial
differential equations. They may also arise as scaling
limit of randomly indexed random walks with dependent jumps.
These results provide some analytic and physical interpretations
of $\alpha$-time fractional Brownian motions. These results
show that $\alpha$-time fBm can serve as useful stochastic model
in various scientific areas.

\subsection{PDE connections}

%{\color{red}Results in this section are for $W(Y(t))$ where $W$ is fractional Brownian motion in $\rd$ of index $0<H<1$ with independent components that are identically distributed. We also assume that $Y(t)$ is real valued stable process independent of $W$.}{\color{blue} If we still work with the definition in the introduction we need to work with the density of $W(Y(t))$
%$$
%q(t,x)= \prod_{i=1}^d 2\int_0^\infty\frac{e^{-\frac{x_i^2}{2s_i^{2H}}}}{\sqrt{2\pi s_i^{2H}}}
%p_t(s_i)ds_i.
%$$
%I am not sure what types of PDEs can be solved by this
%}

%Below we will consider the case $d=1$.

The domain of the infinitesimall generator $\mathcal{A}$ of
 a semigroup $T(t)$ defined on a Banach space $\mathcal{H}$ is the set
 of all $\varphi\in \mathcal{H}$ such that the limit
 $$\lim _{t\to 0}\frac{T(t)\varphi(x)-\varphi(x)}{t}$$
exists in the strong norm of $\mathcal{H}$.

Let $\Delta =\sum_{j=1}^{d}\frac{\partial^2 }{\partial x_j^2}$
be the Laplacian operator, and let $\delta(x)$ be the Dirac-delta function.
The density of Brownian motion $W$ in $\R^d$ is  $f(t,x)=\frac{1}
{(2\pi t)^{d/2}}e^{-|x|^2/2t}$. Let
$$T(t)\varphi(x)=\int_\rd f(t, x-y)\varphi(y)dy
$$
be the semigroup of Brownian motion on $L^2(\rd)$.
Then the generator of $T(t)$ is $\Delta$ with the domain $Dom (\Delta)
=\{\varphi \in L^2(\rd): \nabla \varphi\in L^2(\rd)\}$, where $\nabla \varphi$ is the
weak derivative of $\varphi$. See Section 31 in Sato \cite{sato} for more
details and semigroups on other Banach spaces.
Let $\varphi$ be a function in the domain of the Laplacian.
Then the function $
u(t,x)=T(t)\varphi(x)
$ is  a solution of the heat equation
\begin{equation}
\begin{split}
\frac{\partial}{\partial t}u(t,x)&=\frac12\Delta u(t,x),  \ \ \ t > 0,\, x\in \rd,\\
u(0,x)&=\varphi (x),\ \ \  x\in \rd.
\end{split}\end{equation}

In the case $H= \frac 1 2$ and $Y(t)$ is stable subordinator of index
$\beta/2$, $0<\beta\leq 2$ with $\E(e^{-sY(t)})=e^{-ts^{\beta/2}}$,
$W(Y(t))$ is a symmetric stable process of index $\beta$ in $\R^d$. The
density of  $W(Y(t))$ is given by
$$
q(t,x)=\int _0^\infty f(s,x) p_t(s)\, ds
=\int _0^\infty \frac{e^{-|x|^2/2 s}}{(2\pi s)^{d/2}}p_t(s)ds,
$$
where $p_t(s)$ is the density of $Y(t).$ Then the function
$$
u(t,x)=\E_{x}[\psi(W(Y(t)))]=\int_0^\infty [T(s)\psi(x)]p_t(s)ds
$$
is  a solution of
\begin{eqnarray}
\frac{\partial}{\partial t}u(t,x)\ & = &
-2^{-\beta/2}(-\Delta)^{\beta/2} u(t,x),
\ \ \ \ \  \ t>0, \ \ x\in \rd\label{symmetric-stable-pde}\\
u(0,x)& = & \psi(x), \ \ \ \ \ x\in \rd,\nonumber
\end{eqnarray}
where $-(-\Delta)^{\beta/2}$ is the fractional Laplacian
with Fourier transform
$$
\int_{\R^d}e^{i\l k, x \r}[-(-\Delta)^{\beta/2}\psi(x)]dx=-|k|^\beta
\int_{\rd}e^{i \l k , x\r}\psi(x)dx,
$$
for functions  $\psi$ in the domain of the fractional Laplacian,
see \cite[Theorem 31.5 and Example 32.6]{sato}.

For the case of $H=\frac12$ and $\alpha=2$,
Allouba and Zheng \cite{allouba1} and  DeBlassie \cite{deblassie}
showed that, for any function $\varphi$ in the domain of the Laplacian,
the function $u(t,x)=\E_{x}[\varphi(W(Y(t)))] $
solves the Cauchy problem
\begin{eqnarray}
\frac{\partial}{\partial t}u(t,x)\ & = &
\frac{{\Delta}\varphi(x)}{\sqrt{8\pi t}}\ + \
\frac18{\Delta}^{2}u(t,x),
\ \ \ \ \  \ t>0, \ \ x\in \rd\nonumber\\
u(0,x)& = & \varphi(x), \ \ \ \ \ x\in \rd.\nonumber
\end{eqnarray}
In this case $u(t,x)=\E_x[\varphi(W(Y(t)))]$ also solves the fractional
Cauchy problem
\begin{eqnarray}
\frac{\partial^{\frac12} }{\partial t^{\frac12} }u(t,x) &=&2^{-3/2}
\Delta u(t,x);  \  \ x\in \rd, \ t>0\nonumber\\
u(0,x)& =& \varphi(x), \ x\in \rd.\nonumber
\end{eqnarray}
Here $\frac{\partial^{\frac12}  }{\partial t^{\frac12}}u(t,\cdot)$
is the Caputo fractional derivative with respect to $t$ of order $\frac12$,
defined by (for fixed $x\in \rd$)
\begin{equation}\label{CaputoDef}
\frac{\partial^{\frac12}}{\partial t^{\frac12}}u(t,x)=\frac{1}
{\sqrt{\pi}}\int_0^t \bigg[ \frac{\partial u(s,x)}{\partial s}\bigg]\frac{ds}
{(t-s)^{\frac12}}=\frac{1}{\Gamma(1-\frac12)}
\int_0^t \bigg[ \frac{\partial u(s,x)}{\partial s}\bigg]\frac{ds}{(t-s)^{\frac12}},
\end{equation}
see \cite{bmn-07}.

For the case $H=\frac1 2$, $\alpha=1$ and $Y$ is a symmetric Cauchy process,
Nane \cite{nane3} showed that $u(t,x)=\E_x[\varphi(W(Y(t)))]$ solves
\begin{eqnarray}
\frac{\partial^{2}}{\partial t^{2}}u(t,x)\ & = &-\frac{\Delta
\varphi(x)}{\pi t}\, - \, \frac14\Delta ^{2} u(t,x),
\ \ \ \ \  \ t>0, \ \ x\in \rd,\label{PDE-CONNECT3}\\
 u(0,x) & = & \varphi(x), \ \ \ \ \ x\in \rd,\nonumber
\end{eqnarray}
where $\varphi$ is a bounded measurable function in the domain of
the Laplacian, with $\frac{\partial^2\varphi} {\partial x_i\partial x_j}$
bounded and H\"{o}lder continuous for all $1\leq i,\ j \leq d$.

Nane \cite{nane3} has also established pde connection
for the case $H=\frac12$, and $\alpha=\frac km$ for relatively
prime integers $k,m$, see Theorem 2.5 in \cite{nane3}.

For the case $0<H<1$, $d=1$ and $\alpha=2$, D'ovidio and
Orsingher \cite{dovidio-orsingher} established the
fact that the density of $Z(t)= W(Y(t))$
$$q(t,x)=2\int_0^\infty \frac{e^{-\frac{x^2}{2s^{2H}}}}{\sqrt{2\pi s^{2H}}}
\frac{e^{-\frac{s^2}{2t}}}{\sqrt{2\pi t}}ds$$
is a solution of the first order PDE
\begin{equation}\label{pde-h-alpha-2}
t\frac{\partial q(t,x)}{\partial t}=-\frac H2 \frac{\partial}
{\partial x}(xq(t,x)), \ \ \ t>0,\, x\in \R.
\end{equation}

%{\color{red}Now we consider real-valued $\alpha$-time fractional Brownian motion $Z(t) = W(Y(t))$
%%with $\alpha = 1$. Let $f^H(t,x)$ be the density function of the fractional Brownian
%motion $W$ with Hurst index $H$,}
Let  $p_t(s)$ be the density function of the symmetric
Cauchy process $Y (t)$ and let $\delta (x)$ denote the Dirac-delta function.
The following theorem answers a question in \cite{dovidio-orsingher} and we
prove it and Theorem \ref{thm-pde-h-stable} below for the more general case
that $W$ is a fractional Brownian motion with values in $\R^d$.
\begin{theorem}\label{thm-pde-h-alpha-1}
In the case $0<H<1$ and $\alpha=1$ the density function
$$q(t,x)=2\int_{0}^\infty f^H(s,x)p_t(s)ds=2\int_0^\infty \frac{e^{-\frac{|x|^2}{2s^{2H}}}}
{(2\pi s^{2H})^{d/2}}\frac{t}{\pi(t^2+s^2)}ds$$
 of $W(Y(t))$ solves the PDE
\begin{equation}\label{h-alpha-1}\begin{split}
\frac{\partial^2 q(t,x)}{\partial t^2}
&=-\frac{2H I_{(0,1/2]}(H)}{\pi t}\Delta \delta (x)-H(2H-1) \Delta G_{(2H-2),t} q(t,x)
\\
&\qquad \qquad \ \ \ -H^2 \Delta^2 G_{(4H-2), t}q(t,x), \ \ \ \
x\in \rd, t>0,
\end{split}\end{equation}
where
$$
G_{\gamma, t} q(t,x)=2\int_{0}^\infty
s^\gamma p_t(s) f^H(s,x)ds, \ \ \gamma\neq 0,
$$
and $G_{0,t}$ is the identity operator.
\end{theorem}

An operator similar to the operator $G_{\gamma, t}$ was introduced in
\cite{hahn-kobayashi-umarov}. We refer to their Proposition 3.6 and
Remark 3.7 for some nice properties of that operator. For the case $H\neq 1$,
it might be a challenging problem to find the right class of functions
$\phi$ and establish the Cauchy problem that is solved by  $u(t,x)=
\E_x(\phi(W(Y(t))))$. This is due to the fact that $v(t,x)=\E_x(\phi(W(t)))$
is not a semigroup on a Banach space. The general theory of semigroups and
their generators will not apply in this case

\begin{proof}
Recall that the density function of symmetric Cauchy process $Y(t)$ is
$$
p_t(s)=\frac{t}{\pi(t^2+s^2)}, \ \ t\geq 0,\ s\in \R.
$$
Since
$$
\frac{\partial^{2}}{\partial
t^{2}}p_t(s)=\frac{-2t(3s^2-t^2)}{(t^2+s^2)^3}
$$
and for $t>0$
$$
2\int_{0}^{\infty}f^H(s,x)\bigg|\frac{\partial^{2}}{\partial t^{2}}p_t(s)
\bigg|ds=2\int_{0}^{\infty}\frac{e^{-\frac{|x|^2}{2s^{2H}}}}
{(2\pi s^{2H})^{d/2}}\bigg|\frac{-2t(3s^2-t^2)}{(t^2+s^2)^3}\bigg|ds<\infty,
$$
we apply the Dominated Convergence Theorem to verify the following interchange
of the second derivative in $t$:
\begin{equation}\label{Eq:27}
\frac{\partial^{2}}{\partial t^{2}}q(t,x) =
2\int_{0}^{\infty}f^H(s,x)\frac{\partial^{2}}{\partial t^{2}}p_t(s)ds.
\end{equation}

By using  integration by parts to  (\ref{Eq:27}) and the facts
\begin{equation}\label{pde-cauchy-fractional-bm}
\begin{split}
&\bigg(\frac{\partial^{2}}{\partial s^{2}}+\frac{\partial^{2}}{\partial
t^{2}}\bigg)p_t(s)=0;\\
&\frac{\partial}{\partial s}f^H(s,x)=H s^{2H-1}
\Delta f^H (s,x),
\end{split}
\end{equation}
we derive
\begin{eqnarray}
\frac{\partial^{2}}{\partial t^{2}}q(t,x)& =&
-2\int_{0}^{\infty}f^H(s,x)\frac{\partial^{2}}{\partial s^{2}}p_t(s)ds\nonumber\\
&=&-2f^H(s,x)\frac{\partial}{\partial s}p_t(s)\Big|_{0}^\infty +
2\int_{0}^{\infty}\frac{\partial}{\partial s}f^H(s,x)\frac{\partial}{\partial s}p_t(s)ds\nonumber\\
&=&2p_t(s)\frac{\partial}{\partial s}f^H(s,x)\Big|_{0}^\infty -
2\int_{0}^{\infty}p_t(s)\frac{\partial^{2}}{\partial s^{2}}f^H(s,x)  ds\nonumber\\
&=&2p_t(s)\frac{\partial}{\partial s}f^H(s,x)\Big|_{0}^\infty+2\int_{0}^{\infty}
p_t(s)\frac{\partial}{\partial s}\bigg( H s^{2H-1}\Delta f^H(s,x) \bigg) ds\nonumber\\
&=&-\frac{2H I_{(0,1/2]}(H)}{\pi t}\Delta \delta (x)\nonumber\\
&&-2\int_{0}^{\infty}p_t(s)\bigg(H(2H-1)s^{2H-2}\Delta f^H(s,x)+H^2s^{4H-2}\Delta^2f^H(s,x)\bigg)ds.\nonumber\\
&=&-\frac{2H I_{(0,1/2]}(H)}{\pi t}\Delta \delta (x)\nonumber\\
&&-\Delta 2\int_{0}^{\infty}p_t(s)H(2H-1)s^{2H-2}f^H(s,x)ds +\Delta^2 2\int_{0}^{\infty}p_t(s) H^2s^{4H-2}f^H(s,x)ds.\nonumber
\end{eqnarray}
the last line follows by the dominated convergence theorem.
In the above we have used that
\begin{equation}\begin{split}
\lim _{s\to 0}f^H(s,x)\frac{\partial}{\partial s}p_t(s)&=0,\\
\lim _{s\to \infty}f^H(s,x)\frac{\partial}{\partial s}p_t(s)&=0,\\
\lim _{s\to \infty}p_t(s)\frac{\partial}{\partial s}f^H(s,x)&=0,
\end{split}
\end{equation}
and that
\begin{equation}
\begin{split}
\lim _{s\to 0}p_t(s)\frac{\partial}{\partial s}f^H(s,x)
&=\lim _{s\to 0}\frac{H}{\pi t}s^{2H-1}\Delta f^H(s,x)\nonumber\\
&=\left\{
\begin{array}{c}
  0\quad\qquad\qquad\quad\mbox{ if } \, H>\frac12, \\
  \frac{H}{\pi t}\Delta \delta(x) \quad\mbox{ if } \, 0<H\leq \frac12.
\end{array}
\right.\nonumber
\end{split}
\end{equation}
This finishes the proof of \eqref{h-alpha-1}.
\end{proof}
\begin{remark}
After we submitted our paper, we learned that Beghin et al. \cite{beghin-et-al}
has established that the density  $W(Y(t))$ in the case  $d=1$, $0<H<1, \alpha=1$
$$
q(t,x)=2\int_0^\infty\frac{e^{-\frac{x^2}{2s^{2H}}}}{\sqrt{2\pi s^{2H}}}\frac{t}{\pi(t^2+s^2)}ds
$$
solves
\begin{equation}\label{second-order-h-alpha-1}
\frac{\partial^{2}}{\partial t^{2}}q(t,x)=-\frac{1}{t^2}\bigg[H(H-1)
\frac{\partial}{\partial x}x-H^2\frac{\partial^2}{\partial x^2}x^2\bigg]
q(t,x)-\frac{2H I_{(0,1/2]}(H)}{\pi t}\frac{\partial^2 \delta (x)}{\partial x^2}.
\end{equation}
It is interesting to compare the equations \eqref{h-alpha-1} and
\eqref{second-order-h-alpha-1}.
\end{remark}

Let $0<\beta=\frac km <2$ for $k,m$ relatively prime integers,
and let $Y(t)$ be a stable subordinator of index $\beta/2$.
In this case the density $p_t(s)$ is a solution of
\begin{equation}\label{pde-stable-sub}
\frac{\partial^{2m}}{\partial t^{2m}}p_t(s)
=\frac{\partial^k}{\partial s^k}p_t(s), \ \ s, t>0,
\end{equation}
see Lemma 3.1 in \cite{deblassie2}. We have the next theorem which
gives an extension of the PDE in \eqref{symmetric-stable-pde}.
\begin{theorem}\label{thm-pde-h-stable}
Let $0<H<1$, $\beta=\frac 1m$ for $m=2,3,\cdots$. Let $Y(t)$
be a stable subordinator of index $\alpha= \frac{\beta} 2$.
Then the density $q(t,x)=\int_0^\infty f^H(s,x)p_t(s)ds$
of $W(Y(t))$ is a solution of
\begin{equation}
\frac{\partial^{2m}}{\partial t^{2m}}\,q(t,x)
=-H\Delta V_{(2H-1),t}\,  q(t,x),
\end{equation}
where $V_{\gamma, t}\, q(t,x)=
\int_0^\infty s^\gamma p_t(s) f^{H}(s,x)
p_t(s)ds$ for $\gamma\neq 0$ and $V_{0,t}$ is the identity operator.
\end{theorem}
\begin{proof}
The proof follows by integration by parts as in the proof of Theorem
\ref{thm-pde-h-alpha-1}, and by using \eqref{pde-stable-sub} with $k=1$.
\begin{eqnarray}
\frac{\partial^{2m}}{\partial t^{2m}}\,q(t,x)& = &
\int_{0}^{\infty}f^H(s,x)\frac{\partial^{2m}}{\partial t^{2m}}p_t(s)ds\nonumber\\
&= &\int_{0}^{\infty}f^H(s,x)\frac{\partial}{\partial s}p_t(s)ds\nonumber\\
&=&f^H(s,x)p_t(s)\Big|_{0}^\infty -\int_{0}^{\infty}\frac{\partial}
{\partial s}f^H(s,x)p_t(s)ds\nonumber\\
&=&f^H(s,x)p_t(s)\Big|_{0}^\infty-\int_{0}^{\infty}p_t(s)
\bigg( H s^{2H-1}\Delta f^H(s,x) \bigg)\, ds\nonumber\\
&=&-H
\int_{0}^{\infty}p_t(s) s^{2H-1}\Delta f^H(s,x)\,ds,\nonumber\\
&=&-H
\Delta\int_{0}^{\infty}p_t(s) s^{2H-1} f^H(s,x)\,ds,\nonumber
\end{eqnarray}
the last line follows by dominated convergence theorem.
See equations (2.7)--(2.10) in \cite{deblassie2} to show that the
boundary terms are all zero.
\end{proof}

Letting $H=\frac12$ in Theorem \ref{thm-pde-h-stable}, the density
$q(t,x)$ of symmetric stable process $W(Y(t))$ of index $\alpha=\frac 1m$
is a solution of
 \begin{equation}\label{symmetric-stable-1-m-pde}
\frac{\partial^{2m}}{\partial t^{2m}}\,q(t,x)
=-\frac12\Delta \,q(t,x).
\end{equation}
The equations \eqref{symmetric-stable-pde} and
\eqref{symmetric-stable-1-m-pde} should be compared.
This result is a special case of the following result
stated in  Nane \cite[Lemma 3.2]{nane3} that is due to
DeBlassie \cite{deblassie2} originally:
Let  $H=\frac12$ and $0<\alpha=\frac km <2$, where $k, m$ are
relatively prime. Let $Y$ be a stable subordinator of
index $\alpha/2$. In this case $W(Y(t))$ is a symmetric
stable process of index $\alpha=\frac km$. Then
the density $q(t,x)$ of $W(Y(t))$ is a solution of
\begin{equation}\label{symmetric-stable-k-m-pde}
\frac{\partial^{2m}}{\partial t^{2m}}\,q(t,x)
=\frac{1}{2^k}(-\Delta)^k\,q(t,x), \ \ s, t>0.
\end{equation}
We can work out a similar connection for the case $0<H<1$ and
$\alpha=\frac km$ ($k$ and $m$ are relatively prime integers)
by using integration by parts, \eqref{pde-cauchy-fractional-bm},
and \eqref{symmetric-stable-k-m-pde}, which extends the
PDE connection in Nane \cite[Theorem]{nane3} for the
case $H=\frac12, \alpha=\frac km$.

Alternatively, for the case $0<H<1$, $\alpha=\frac km$, $k,m$
relatively prime integers, we have $W(Y(t))=W(B(U(t)))$, where
$B$ is a Brownian motion running twice the speed of standard Brownian
motion and $U$ is a stable subordinator of index $\alpha/2=\frac{k}{2m}$.
In this case using the methods of Theorem \ref{thm-pde-h-alpha-1},
equation \eqref{pde-h-alpha-2} for the density of $W(B(t))$ and equation
\eqref{pde-stable-sub} for the density of $U(t)$ we can obtain the PDE
solved by the density of $W(Y(t))$.

For many other PDE connections of  different types of
subordinate processes, we refer to \cite{allouba1, bmn-07,
bmn-09, dovidio-orsingher, nane3, OB1}.

\subsection{Scaling limits of randomly indexed random walks}

Now we prove that the $\alpha$-time fBm $ Z = \big\{W\big(Y(t)\big), t \ge 0\big\}$,
where $W$ is $H$-fractional Brownian motion with values in $\R$ and $\{Y(t),\,t\ge0\}$
is a stable subordinator of index $\alpha\in(0,\,1)$, can be
approximated weakly in the Skorohod space
$D\left([0,\,\infty),\,\R\right)$ by normalized
partial sums of randomly indexed random walks with dependent jumps.

Let $\{\xi_n,\,n\in\Z\}$ be a sequence of i.i.d. random
variables with $\E[\xi]=0$, $\E[\xi^2]=1$, and let $\{a_n,\,n\in\Z_+\}$
be a sequence of real numbers such that
\[\sum_{n=0}^\infty a_n^2<\infty.\]
We consider the linear stationary process $\{X_n,\,n\in\N\}$ defined by
\begin{equation}\label{Eq:WC1}
X_n=\sum_{j=0}^\infty a_j\xi_{n-j},\quad n\in\N.
\end{equation}
Davydov \cite{davydov} was the first to study the weak convergence
of normalized partial sums of $\{X_n,\,n\in\N\}$ to fractional
Brownian motion. The following result is taken from Whitt
\cite[Theorem 4.6.1]{whitt}.
\begin{lemma}\label{Thm:DavWhitt}
Let $\{X_n,\,n\in\N\}$ be the linear stationary process defined
by (\ref{Eq:WC1}), and let $S_n=X_1+\cdots+X_n$. If
\begin{equation}\label{Eq:WC2}
{\rm Var}(S_n)= n^{2H}L(n),\quad n\in \N
\end{equation}
for some $H\in(0,\,1)$, where $L(\cdot)$ is a slowly varying function,
and
\begin{equation}\label{Eq:WC3}
\E\left[|S_n|^{2a}\right]\le K\cdot \left(\E[S_n^2]\right)^a
\end{equation}
for some constants $a>1/H$ and $K>0$, then
\begin{equation}\label{Eq:WC4}
\left\{\frac1{n^H\sqrt{L(n)}}S_{\lfloor nt \rfloor},\,t\ge 0\right\}
\Rightarrow \big\{W(t),\,t\ge0\big\}
\end{equation}
in the $J_1$-topology on $D\left([0,\,\infty),\,\R\right),$
where $W$ is a fractional Brownian motion with Hurst index $H.$
\end{lemma}

\begin{example}\label{Example:WC1}{\rm
As in \cite[pp.123--124]{whitt} we take $a_j=c j^{-\gamma}$ for
some constants $c\in\R\backslash\{0\}$ and
$\gamma \in(\frac12,\,1),$ then it can be verified that
\begin{equation}\label{Eq:WC5}
{\rm Var}(S_n)\sim c_1 n^{3-2\gamma}\quad\mbox{as } n\to\infty,
\end{equation}
where
\[
c_1=\frac{2c^2\,\Gamma(1-\gamma)\Gamma(2\gamma-1)}{\Gamma(\gamma)
(3-2\gamma)^2}.
\]
By applying Lemma \ref{Thm:DavWhitt}, we have that
\begin{equation}\label{Eq:WC6}
\left\{\frac1{\sqrt{c_1}\,n^H}S_{\lfloor nt \rfloor},\,t\ge 0\right\}
\Rightarrow \{W(t),\,t\ge0\},
\end{equation}
in the $J_1$-topology on $D\left([0,\,\infty),\,\R\right),$
where $H=\frac{3-2\gamma}{2}.$}
\end{example}

\begin{theorem}\label{Thm:WC}
Let $\{X_n, n \ge 1\}$ be the linear stationary process
defined by (\ref{Eq:WC1}) and satisfies  (\ref{Eq:WC2})
and (\ref{Eq:WC3}) in Lemma \ref{Thm:DavWhitt}.
Let $\{J_n, n \ge 1\}$ be a sequence of i.i.d. random variables
also independent of $\{\xi_n,\, n\in\Z_+\},$
which belongs to the domain of attraction of some
stable law $Y$ with index $\alpha\in(0,\,1)$ and $Y>0$ a.s.
Denote by $\{b_n, n \ge 1\}$ a sequence of positive
numbers such that $b_nT_n\Rightarrow Y$, where
\[T_n=J_1+\cdots+J_n,\quad \forall n\ge 1.\]
Then as $c \to \infty$
\begin{equation}\label{Eq:WC7}
\left\{\frac1{\big(b(c)\big)^{-H}\sqrt{L\big(b(c)^{-1}\big)}}
S_{\lfloor T(ct)\rfloor},\,t\ge 0\right\}
\Rightarrow \left\{W\big(Y(t)\big),\,t\ge0\right\}
\end{equation}
in the $J_1$-topology on $D\left([0,\,\infty),\,\R\right),$
where $b(c)=b_{\lfloor c\rfloor}$ and $T(s)=T_{\lfloor s \rfloor}.$
\end{theorem}
\begin{proof}\
For  $t>0,$ let $b(t)=b_{\lfloor t\rfloor}$. Then $b(t)=t^{-1/\alpha}\ell(t),$
for some slowly varying function $\ell(t)$
at infinity. It follows from Theorem 4.5.3 in Whitt \cite{whitt} that
\begin{equation}\label{Eq:WC8}
\left\{b(c)T(ct),\,t\ge 0\right\}\Rightarrow \{Y(t),\,t\ge0\}
\end{equation}
in the $J_1$-topology on $D\left([0,\,\infty),\,\R_+\right),$
where $\{Y(t),\, t\ge 0\}$ is a stable subordinator with
index $\alpha.$

Notice that $\{X_n,\,n\ge 1\}$ and $\{J_n\}$ are independent,
we derive from Lemma \ref{Thm:DavWhitt} and \eqref{Eq:WC8} that
as $c \to \infty$
\begin{equation}\label{Eq:WC9}
\left\{\Big(\frac1{c^H\sqrt{L(c)}}S_{\lfloor ct\rfloor},\,b(c)T(ct)\Big),
\,t\ge 0\right\}
\Rightarrow \Big\{\big(W(t),\,Y(t)\big),\,t\ge0\Big\}
\end{equation}
in the $J_1$-topology on $D\left([0,\,\infty),\,\R\right)\times
D\left([0,\,\infty),\,\R_+\right).$

Since, for the limiting processes in (\ref{Eq:WC9}), the sample
function $W(t)$ is continuous and $Y(t)$
is strictly increasing, the conclusion of Theorem \ref{Thm:WC}
follows from Theorem 13.2.2 in Whitt \cite{whitt}.
\end{proof}

\section{Existence of local times }
\label{sec:Exist}

Let $X= \{X(t), t\geq 0\}$ be an $\alpha$-time fractional Brownian
motion in $\RR{R}^{d}$ defined by (\ref{multi-alpha}). In this
section, we study the existence of local times
$$
L=\{ L(x,B): \ x\in \RR{R}^{d}, B\in \mathcal{B}(\RR{R}_{+})\}
$$
of $X$, where $\mathcal{B}(\RR{R}_{+})$ is the Borel
$\sigma$-algebra of $\RR{R}_{+}$. In Section \ref{sec:continuity},
we will establish joint continuity and sharp H\"older conditions in
the set variable for the local times.

We recall briefly the definition of local times. For an extensive
survey, see Geman and Horowitz \cite{GH80}. Let $X: \R\to \R^d$ be
any Borel function and let $B\subset \RR{R}$ be a Borel set. The
occupation measure of $X(t)$ on $B$ is defined by
\begin{equation}\label{occup-measure}
\mu_{B}(A)=\lambda_{1}\{t\in B: \ \ X(t)\in A\}
\end{equation}
for all Borel sets $A\subset \RR{R}^{d}$, where $\lambda_{1}$ is the
Lebesgue measure on $\R$. If $\mu_{B}$ is absolutely continuous with
respect to the Lebesgue measure $\lambda_{d}$ on $\RR{R}^{d}$, we
say that $X$ has a local time on $B$ and define its local time
$L(x,B)$ to be the Radon-Nikodym derivative of $\mu_{B}$. If
$B=[0,t]$, we will simply write $L(x,B)$ as $L(x,t)$. If $I = [0,
T]$ and $L(x, t)$ is continuous as a function of $(x, t) \in
\R^d \times I$, then we say that $X$ has a jointly continuous
local time on $I$. In this latter case, the set function
$L(x, \cdot)$ can be extended to be a finite Borel measure
on the level set
\[
X^{-1}_I(x) =\{t \in I: X(t) = x\}.
\]
See Adler \cite[Theorem 8.6.1]{adler}. This fact has been used by many
authors to study fractal properties of level sets, inverse image and
multiple times of stochastic processes. Related to our paper, we
mention that Xiao \cite{xiao} and Hu \cite{Hu99} have studied the
Hausdorff dimension, and exact Hausdorff and packing measure of the
level sets of iterated Brownian motion, respectively.

An $\alpha$-stable L\'evy process $Y = \{Y(t), t \ge 0\}$ with
values in $\R$ is a stochastically continuous process with
stationary independent increments, $Y(0) = 0$, and characteristic
exponent $\psi$ given by
\begin{equation}\label{Eq:psi}
\begin{split}
\psi(z)& =- \sigma  |z|^{\alpha}\left( 1-i \beta\, {\rm sgn} (z)\tan
\frac{\pi\alpha}{2}\right), \ \ \ \ \alpha \neq 1;\\
\psi(z) &= -\sigma |z|\left( 1+i\frac{2}{\pi}\beta\, {\rm sgn}(z)
\ln (|z|)\right), \ \ \ \ \alpha=1,
\end{split}
\end{equation}
where $0<\alpha \leq 2$, $\sigma>0$ and $-1\leq \beta \leq 1$ are
fixed constants (we have tacitly assumed that there is no drift
term). See Bertoin \cite{bertoin} and Sato \cite{sato} for a systematic
accounts on L\'evy processes and stable laws, respectively.

Throughout this paper, we assume that $Y= \{Y(t), t \ge 0\}$
is strictly stable. That is, we assume $\beta = 0$ in
(\ref{Eq:psi}) when $\alpha = 1$ so the asymmetric Cauchy process
is excluded. Strictly stable L\'evy processes of index $\alpha$ are
$(1/\alpha)$-self-similar. Recall that, for $t > 0$, $p_t(x)$ is the
density function of the random variable $Y(t)$. It is a bounded
continuous function with the following scaling property:
\begin{equation}\label{Eq:scaling}
p_t(x) = p_{rt}(r^{1/\alpha} x)\, r^{1/\alpha} \ \ \ \hbox{  for
every $r
> 0$}.
\end{equation}

As discovered in Taylor \cite{Taylor67}, it is natural to
distinguish between two types of strictly stable processes:
those of \emph{Type A}, and those of \emph{Type B}.
A strictly stable process, $Y$, is of \emph{Type A}, if
\[
    p_t(x) > 0, \qquad \forall t > 0, \, x \in \R;
\]
all other stable processes are of  \emph{Type B}. Taylor
\cite{Taylor67} has shown that if $\alpha \in (0, 1)$ and  $Y$ is of
\emph{Type B}, then either $Y$ or $-Y$ is a subordinator, while all
other strictly stable processes of index $\alpha \ne 1$  are of
\emph{Type A}. Hence, without loss of generality, we will assume $Y$
is either a strictly stable process of type $A$, or a subordinator. It
will be shown that the properties of local times of $\alpha$-time
fractional Brownian motion $X$ in $\R^{d}$ depends on the type of $Y$.

The following existence theorem for the local time of $X$ is easily
proved  by using the Fourier analysis (see, e.g., Berman
\cite{Berman69}, Geman and Horowitz \cite{GH80} or Kahane
\cite{kahane}).

\begin{theorem}\label{fourier}
Let $X=\{ X(t),\, t \ge 0\}$ be an $\alpha$-time fractional Brownian
motion in $\R^d$. Then for any $T>0$, $X$ has a local time $L(x,T)
\in L^2(\P \times \lambda_d)$ almost surely if and only if $d<\alpha
/H$.
\end{theorem}

\begin{proof}
Let $\mu_{[0,T]}$ be the occupation measure of $X$ on $[0, T]$ defined by
(\ref{occup-measure}). Then its  Fourier transform can be written
as
$$
\widehat{\mu}_{[0,T]}(u)=\int_{0}^{T}\exp (i\l u, X(t) \r) \, dt,
$$
where $\l \cdot ,\, \cdot \r$ is the ordinary scalar product in
$\RR{R}^{d}$. It follows from Fubini's theorem that
\begin{equation}\label{plancharel}
\E\int_{\RR{R}^{d}}|\widehat{\mu}_{[0,T]}(u)|^{2}du= \int_{0}^{T}
\int_{0}^{T}\int_{\RR{R}^{d}}\E \exp(i \l u, X(t)-X(s)\r)\, dudsdt.
\end{equation}
To evaluate the characteristic function in (\ref{plancharel}), we
assume $0<s<t$ (the other case is similar) and note that the density
of $(Y(t),Y(s))$ is given by
$$
p_{s,t}(x,y)=p_{s} (x)p_{t-s} (y-x).
$$

Since $X_{1},\cdots, X_{d}$ are independent copies of
$Z=\{W(Y(t)), t \ge 0\}$, we have
\begin{equation} \label{laplace-trans}
\begin{split}
\E\exp(i \l u,X(t)-X(s)\r) &= \prod_{k=1}^{d} \E\exp
\big(iu_{k}(W(Y(t))-W(Y(s))) \big) \\
&= \prod_{k=1}^{d} \int_{\RR{R}} \int_{\RR{R}} \ \E \exp
\big(iu_{k}(W(y)-W(x)) \big) p_{s,t}(x,y)\, dxdy \\
&= \prod_{k=1}^{d} \int_{\RR{R}} \int_{\RR{R}} \exp
\Big(-\frac{u_{k}^{2}} {2}|y-x|^{2H} \Big)\, p_{s,t}(x,y)\, dxdy \\
&=  \prod_{k=1}^{d}\int_{\RR{R}}
\exp\Big(-\frac{u_{k}^{2}}{2}|z|^{2H} \Big) p_{t-s} (z)\, dz.
\end{split}
\end{equation}
To evaluate the integrals with respect to $u$, we make a change of
variables to get
\begin{equation}\label{Eq:F2}
 \int_{\RR{R}}\exp\Big(-\frac{u_{k}^{2}}{2}\, |z|^{2H} \Big)\, du_{k}
= |z|^{-H} \int_{\RR{R}}\exp\Big(-\frac{u_{k}^{2}}{2}\Big)\, du_{k}.
\end{equation}
It follows from (\ref{plancharel}), (\ref{laplace-trans}) and
(\ref{Eq:F2}) that
\begin{equation}\label{Eq:F3}
\begin{split} \E\int_{\RR{R}^{d}}|\widehat{\mu}_{[0,T]}(u)|^{2}\, du
&= \int_{0}^{T}\int_{0}^{T}\prod_{k=1}^{d}\left(\int_{\RR{R}}\exp
\Big(-\frac{u_{k}^{2}}{2}\Big)\, du_{k}
\int_{\RR{R}}|z|^{-H} p_{|t-s|} (z)\, dz\right) ds dt \\
&=
(2 \pi)^{d/2}\left( \int_{\RR{R}}|z|^{-H} p_{1}(z)dz\right)^{d}
\int_{0}^{T}\int_{0}^{T}\frac{1}{|t-s|^{dH/\alpha}}\, dsdt.
\end{split}
\end{equation}
In the above, we have used the fact that $p_{|t-s|} (z) =
|t-s|^{-1/\alpha}p_{1}  (|t-s|^{-1/\alpha}z)$ and another change of
variables.

The last integral in (\ref{Eq:F3}) is finite if and only if
$dH/\alpha< 1$. Hence $\widehat{\mu}(\cdot)\in L^{2}(\P \times
\lambda_d)$ if and only if $dH/\alpha< 1$. Therefore, Theorem
\ref{fourier} follows from Plancherel's theorem (see also Theorem
21.9 in Geman and Horowitz \cite{GH80}).
\end{proof}

We can express the local time $L(t,x)$ as the inverse Fourier
transform of $\widehat{\mu}_{[0, T]}(u)$, namely
\begin{equation}\label{local-time-def}
\begin{split}
L(t,x)&= \left(\frac{1} {2\pi} \right)^{d} \int_{\RR{R}^{d}}
\exp (-i \l u, x\r)\, \widehat{\mu}_{[0, t]}(u)\, du \\
&=\left(\frac{1} {2\pi}\right)^{d} \int_{0}^{t} \int_{\RR{R}^{d}}
\exp \big(-i \l u, x \r \big)  \, \exp \big( i \l u, X(s) \r
\big)\,du\, ds.
\end{split}
\end{equation}
It follows from (\ref{local-time-def}) that for any $x,w\in
\RR{R}^{d}$, $B\in \mathcal{B}(\RR{R}_{+})$ and all integers $n\geq
1$, we have
\begin{equation} \label{local-time-moment}
\begin{split}
\E[L(x,B)]^{n}&=
(2\pi)^{-nd}\int_{B^{n}}\int_{\RR{R}^{nd}}\exp \bigg(-i \sum_{j=1}^{n}
 \l u_{j},x \r \bigg)\\
& \qquad \quad  \times \E\exp\bigg( i\sum_{j=1}^{n} \l u_{j},X(t_{j})
\r\bigg)\, d\bar{u} \, d\bar{t}
\end{split}
\end{equation}
and for all even integers $n \ge 2$,
\begin{equation} \label{local-time-diff-moment}
\begin{split}
& \E[L(x+w,B)-L(x,B)]^{n}\\ &=(2\pi)^{-nd} \int_{B^{n}}
\int_{\RR{R}^{nd}} \prod_{j=1}^{n} \Big(\exp(-i \l u_{j},
x+w \r)-\exp(-i \l u_{j},x \r) \Big) \\
& \qquad \qquad \qquad \quad \times \E\exp\bigg(i\sum_{j=1}^{n}
\l u_{j}, X(t_{j}) \r \bigg)\, d\bar{u}\, d\bar{t},
\end{split}
\end{equation}
where $\bar{u}=(u_{1}, \cdots, u_{n})$, $\bar{t}=(t_{1}, \cdots,
t_{n})$ and each $u_{j}\in \RR{R}^{d}$, $t_{j}\in B$ $(j=1,\cdots,
n)$. In the coordinate notation we then write
$u_{j}=(u_{j}^{1},\cdots ,u_{j}^{d})$. For details in deriving the
equations (\ref{local-time-moment}) and
(\ref{local-time-diff-moment}), see Geman and Horowitz \cite{GH80}.

\section{Joint continuity and H\"older conditions}
\label{sec:continuity}

In this section, we establish the joint continuity and sharp
H\"{o}lder conditions in the set variable for the local times of
$d$-dimensional $\alpha$-time fractional Brownian motion $X$. Then
we apply these results to study the irregularities of the sample
paths of $X(t)$.

We use methods which are similar to those in Ehm \cite{ehm} and
Xiao \cite{xiao2, xiao}. The following Lemmas \ref{moments},
\ref{moments-difference} and
\ref{moments-2} give the crucial estimates for the moments of the
local time of $\alpha$-time fractional Brownian motion. Note that
the estimates in the case $Y$ is of type A  [i.e., (\ref{moment-1})
and (\ref{moment-2})] are different from the case when $Y$ is a
stable subordinator [see Lemma \ref{moments-2}].

We need the fact that fractional Brownian motion
$W$ satisfies the property of strong local nondeterminism (SLND), which was proved by Pitt \cite{Pitt78}. More
precisely, for any $y_{1},\cdots, y_{n}\in \RR{R}$,
\begin{equation}\label{nondetermin}
\begin{split}
{\rm Var} \big(W(y_{n})|W(y_{1}),\cdots ,W(y_{n-1}) \big)
\geq K\, \min_{0 \le j \le n-1} |y_{n}-y_{j}|^{2H},
\end{split}
\end{equation}
where $y_{0}=0$ and $K>0$ is an absolute constant.

\begin{lemma}\label{moments}
Let $X = \{X(t),\, t\geq 0\}$ be a $d$-dimensional $\alpha$-time
fractional Brownian motion with $d<\alpha /H$ for which $Y(t)$ is of
type A. For any $h>0$, $B=[0,h]$, $x\in \RR{R}^{d}$, any
integer $n\geq 1$, we have
\begin{equation} \label{moment-1}
\E\big[L(x,B) \big]^{n} \leq K^{n}\, h^{(1-dH/\alpha)n}\,
(n!)^{dH(1+1/\alpha)},
\end{equation}
where $K>0$ is a finite constant depending on $d$, $H$ and $\alpha$ only.
\end{lemma}
\begin{proof}\ Thanks to the strong local nondeterminism (SLND) of
fractional Brownian motion [cf. Eq. (\ref{nondetermin})], Lemma
\ref{mainlemma} and the scaling property of $p_t(x)$ [cf. Eq.
(\ref{Eq:scaling})], the proof of Lemma \ref{moments} follows
along a similar line of the proof of Eq. (2.11) in Xiao \cite{xiao}
with obvious modifications. We omit the details.
\end{proof}

\begin{lemma}\label{moments-difference}
Under the conditions of Lemma \ref{moments}, we have
that for all even integers $n\geq 2$ and
$0<\gamma<\frac{1}{2}\min \{\alpha/(Hd)-1,\,1 -H\}$, $x,w\in \RR{R}^{d} $
\begin{equation}\label{moment-2}
\E[L(x+w,B)-L(x,B)]^{n}\leq K^{n}|w|^{n\gamma}h^{n(1-(d+\gamma)H/\alpha)}
(n!)^{d+\gamma+\frac{H(d+2\gamma)}{\alpha}},
\end{equation}
where $K>0$ is a finite constant depending on $d$, $H$, $\gamma$
and $\alpha$ only.
\end{lemma}

\begin{proof}\ Even though the arguments for proving Lemma
\ref{moments-difference} are similar to that of Eq. (2.12) in
Xiao \cite{xiao}, several essential modifications are needed.

By (\ref{local-time-diff-moment}) and the elementary inequality
$$
|e^{iu}-1|\leq 2^{1-\gamma}|u|^{\gamma}\ \  \hbox{  for all } \ u\in
\RR{R} \hbox{ and }\ 0<\gamma<1,
$$
we see that for any even integer $n\geq 2$ and any $0<\gamma<1$,
\begin{equation} \label{moment-diff-good-1}
\begin{split}
&\E[L(x+w,B)-L(x,B)]^{n}\\
& \leq |w|^{n\gamma} \int_{B^{n}} \int_{\RR{R}^{nd}}
\prod_{j=1}^{n}|u_{j}|^{\gamma}
 \E\exp\bigg( i\sum_{j=1}^{n} \l u_{j}, X(t_{j}) \r
\bigg)\, d\bar{u} d\bar{t}.
\end{split}
\end{equation}
By making the change of variables $t_{j}=h s_{j}$, $j=1,\cdots,n$
and $u_{j}= h^{-H/\alpha}v_{j}$,  $j=1,\cdots,n$ and changing the
letters $s,v$ back to $t,u$, the self-similarity of $X$ implies that
the right-hand side of (\ref{moment-diff-good-1}) equals
\begin{equation}\label{moment-diff-good-2}
\begin{split}
& |w|^{n\gamma}\, h^{n(1-(d+\gamma)H/\alpha)}
\int_{[0,1]^{n}} \int_{\RR{R}^{nd}} \prod_{j=1}^{n}
|u_{j}|^{\gamma} \E\exp\bigg( i\sum_{j=1}^{n} \l u_{j}, X(t_{j})
\r\bigg)\, d\bar{u}d\bar{t}.
\end{split}
\end{equation}
We fix any distinct $t_{1},\cdots , t_{n} \in [0,1]$ satisfying
\begin{equation}\label{increasing-t}
0=t_{0}<t_{1}<t_{2}<\cdots<t_{n},
\end{equation}
and consider the inside integral in (\ref{moment-diff-good-2}).
Since for any $0<\gamma<1$, $|a+b|^{\gamma}\leq |a|^{\gamma}
+|b|^{\gamma}$, we have
\begin{equation} \label{sum-ineq}
\prod_{j=1}^{n}|u_{j}|^{\gamma} \leq  {\sum}^{\, '}
\prod_{j=1}^{n}|u_{j}^{k_{j}}|^{\gamma},
\end{equation}
where the summation $\sum'$ is taken over all $(k_{1},\cdots,
k_{n})\in \{1,\cdots,d\}^{n}$.

Let us fix a sequence $(k_{1},\cdots, k_{n}) \in
\{1,\cdots,d\}^{n}$, and consider the integral
\begin{equation}
\begin{split}
J&= \int_{\RR{R}^{nd}} \prod_{j=1}^{n} |u_{j}^{k_{j}}|^{\gamma}\,
\E\exp\bigg( i\sum_{j=1}^{n} \l u_{j}, X(t_{j}) \r \bigg)\,
d\bar{u}\\
&=\int_{\RR{R}^{nd}}\prod_{j=1}^{n}|u_{j}^{k_{j}}|^{\gamma}
\E\exp\Bigg( i\sum_{\ell=1}^{d}\sum_{j=1}^{n} u_{j}^{\ell}
W_{\ell}(Y_{\ell}(t_{j}))\Bigg)\, d\bar{u},
\end{split}
\end{equation}
since $X(t_{j})=(W_{1}(Y_{1}(t_{j})),\cdots, W_{d}(Y_{d}(t_{j})))
$.
Now, we condition on $ Y_{\ell}(t_{j})=y_{\ell j}$, $\ell=1,\cdots, d,
 j=1,\cdots,n$.
By independence of the processes $Y_{\ell}$ we have that
the density of
$$
(Y_{\ell}(t_{j})=y_{\ell j}:\ \ell=1,\cdots, d, j=1,\cdots,n)
$$
is given by
\begin{eqnarray*}
&&\widetilde{\bf p}_{t_{1},\cdots, t_{n}}(y_{11},\cdots,y_{1n}, y_{21},\cdots,y_{2n},\cdots,y_{d1},\cdots,y_{dn})\\
&&=\prod_{\ell=1}^{d}\prod_{j=1}^{n}p_{t_{j}-t_{j-1}}(y_{\ell j}-y_{\ell(j-1)}).
\end{eqnarray*}
Let $\bar{t}=(t_1,\cdots, t_n)$ and
$\bar{y}=(y_{11},\cdots,y_{1n},y_{21},\cdots,y_{2n},
\cdots,y_{d1},\cdots,y_{dn})$. By  conditioning we have
\begin{eqnarray}
J&=&\int_{\RR{R}^{nd}}\int_{\RR{R}^{nd}}
\prod_{j=1}^{n}|u_{j}^{k_{j}}|^{\gamma}
\, \E\exp\left( i\sum_{\ell=1}^{d}\sum_{j=1}^{n}
u_{j}^{\ell}W_{\ell}(y_{\ell j}))\right)\,
\widetilde{\bf p}_{\bar{t}}(\bar{y})\, d\bar{u}\,
d\bar{y}\nonumber.
\end{eqnarray}
For any fixed $\bar y \in \R^{nd}$, let
\begin{eqnarray*}
I&=&\int_{\RR{R}^{nd}}\prod_{j=1}^{n}|u_{j}^{k_{j}}|^{\gamma}
\E\exp\left( i\sum_{\ell=1}^{d}\sum_{j=1}^{n}
u_{j}^{\ell}W_{\ell}(y_{\ell j}))\right)\, d\bar{u}\\
&=&\int_{\RR{R}^{nd}}\prod_{j=1}^{n}|u_{j}^{k_{j}}|^{\gamma}
\exp\left( -\frac{1}{2}{\rm Var}\left(\sum_{\ell=1}^{d}\sum_{j=1}^{n}
u_{j}^{\ell}W_{\ell}(y_{\ell j})\right)\right)\, d\bar{u}.
\end{eqnarray*}
Then by a generalized H\"{o}lder's inequality and
Lemma \ref{Lem:Cuzick}, we have
\begin{equation}
\begin{split}
I&\leq \prod_{j=1}^{n}\left[\int_{\RR{R}^{nd}}|u_{j}^{k_{j}}|^{n\gamma}\,
\exp\left( -\frac{1}{2}{\rm Var}\left(\sum_{\ell=1}^{d}\sum_{j=1}^{n}
u_{j}^{\ell}W_{\ell}(y_{\ell j})\right)\right)\, d\bar{u}\right]^{1/n}
\nonumber\\
&\leq  \frac{(2\pi)^{\frac{nd-1}2}}{(\det{\rm Cov}(W_{\ell}(y_{\ell j}),
1\leq \ell\leq d,\ 1\leq j\leq n))^{1/2}}\, \int_{\RR{R}}|v|^{n\gamma}\, e^{-v^{2}/2}\,dv \prod_{j=1}^{n}\frac{1}{\sigma_{k_j, j}^{\gamma}}\nonumber\\
&\leq  \frac{K^{n} (n!)^{\gamma}}{(\det{\rm Cov}(W_{\ell}(y_{\ell j}),1\leq \ell\leq
d,\ 1\leq j\leq n))^{1/2}} \prod_{j=1}^{n}\frac{1}{\sigma_{k_j, j}^{\gamma}}\nonumber\\
& =\frac{K^{n} (n!)^{\gamma}}{\prod_{\ell=1}^{d}(\det{\rm Cov}(W_{\ell}(y_{\ell j}),\
1\leq j\leq n))^{1/2}}\prod_{j=1}^{n}\frac{1}{\sigma_{k_j, j}^{\gamma}}\nonumber,
\end{split}
\end{equation}
where
\begin{equation}
\begin{split}
\sigma_{k_j, j}^{2}&= {\rm Var}(W_{k_j}(y_{k_j,j})\big|\
W_{\ell}(y_{\ell,i}):\ \  \ell\neq k \ \hbox{ or } \
\ell=k_j,\  i\neq j)\\
&= {\rm Var}(W_{k_j}(y_{k_j,j})\big|\ W_{k_j}(y_{k_j,i}):\ \
i=0 \ \hbox{ or }\ i\neq j).
\end{split}
\end{equation}

For any $\ell\in \{1,\cdots, d\}$ and any $y_{\ell,1},
\cdots, y_{\ell,n}$, there exists a permutation $\pi_\ell$
of $\{1,\cdots, n\}$ such that
$$
y_{\ell,\pi_\ell(1)}\leq y_{\ell,\pi_\ell(2)}\leq \cdots \leq y_{\ell,\pi_\ell(n)}.
$$
Hence, if we write $k_j= \ell$, then by SLND of fractional Brownian motion (\ref{nondetermin}),
\begin{eqnarray}
\sigma_{\ell, j}^{2}
&=&{\rm Var} (W_{\ell}(y_{\ell,j})|\ W_{\ell}(y_{\ell,i}):\
\  i=0 \ \hbox{ or }\ i\neq j)\nonumber\\
&\geq &K\, \min \{|y_{\ell, \pi_\ell(j)}-y_{\ell, \pi_\ell(j-1)}|^{2H},\,
|y_{\ell, \pi_\ell(j+1)}-y_{\ell, \pi_\ell(j)}|^{2H} \}.
\end{eqnarray}
Hence
$$
\prod_{j=1}^{n}\frac{1}{\sigma_{k_j, j}^{\gamma}}
\leq K^n\, \prod_{\ell=1}^d\prod_{j=1}^n\frac{1}
{\min \{|y_{\ell, \pi_\ell(j)}-y_{\ell, \pi_\ell(j-1)}|,
\, |y_{\ell, \pi_\ell(j+1)}-y_{\ell, \pi_\ell(j)}|\}^{H\eta_{\ell,j}\gamma}},
$$
where $\eta_{\ell,j}=1$ if $k_j=\ell$ and $\eta_{\ell,j}=0$ otherwise. Note that
\begin{equation}\label{Eq:eta}
\sum_{\ell=1}^d\sum_{j=1}^n\eta_{\ell,j}=n.
\end{equation}
Since
\begin{eqnarray}
&&\prod_{\ell=1}^d\prod_{j=1}^n\frac{1}{\min \{|y_{\ell, \pi_\ell(j)}
-y_{\ell, \pi_\ell(j-1)}|,\, |y_{\ell, \pi_\ell(j+1)}-
y_{\ell, \pi_\ell(j)}|\}^{H\eta_{\ell,j}\gamma}}\nonumber\\
&&\leq \prod_{\ell=1}^d\prod_{j=1}^n\bigg(\frac{1}
{|y_{\ell, \pi_\ell(j)}-y_{\ell, \pi_\ell(j-1)}|^{H\eta_{\ell,j}\gamma} }
+ \frac{1}{|y_{\ell, \pi_\ell(j+1)}-y_{\ell, \pi_\ell(j)}|^{H\eta_{\ell,j}\gamma}}\bigg)\nonumber\\
&&={\sum}^{\,''}\prod_{\ell=1}^d\prod_{j=1}^n\bigg(\frac{1}
{|y_{\ell, \pi_\ell(j)}-y_{\ell, \pi_\ell(j-1)}|^{H\delta_{\ell,j}\gamma} }\bigg),\nonumber
\end{eqnarray}
where the summation $\sum^{\,''}$ is taken over $2^{nd}$ terms
and $\delta_{\ell,j}\in \{0,1,2\}$ and, thanks to  (\ref{Eq:eta}),
\begin{equation}\label{Eq:delta}
\sum_{\ell=1}^d\sum_{j=1}^n\delta_{\ell,j}\leq 2
\sum_{\ell=1}^d\sum_{j=1}^n\eta_{\ell,j}=2n,
\end{equation}
we obtain
\begin{equation}\label{Eq:extra}
\prod_{j=1}^{n}\frac{1}{\sigma_{k_j, j}^{\gamma}}
\leq K^n\,  {\sum}^{\,''}\prod_{\ell=1}^d\prod_{j=1}^n
\bigg(\frac{1}{|y_{\ell, \pi_\ell(j)}-
y_{\ell, \pi_\ell(j-1)}|^{H\delta_{\ell,j}\gamma} }\bigg).
\end{equation}

Now we go back to estimating $J$. Let
$$
\Delta_{\pi_\ell}= \big\{(y_{\ell,1},\cdots, y_{\ell,n}):\
 y_{\ell,\pi_\ell(1)}\leq y_{\ell,\pi_\ell(2)}\leq \cdots
 \leq y_{\ell,\pi_\ell(n)}\big\}.
 $$
Then
\begin{equation}\label{J-estimate}
J\leq K^{n}(n!)^{\gamma}\sum_{\{\pi_\ell\}}{\sum}^{\, ''}\prod_{\ell=1}^d
\Bigg[\int_{\RR{R}^n\cap \Delta_{\pi_\ell}}\prod_{j=1}^{n}
\bigg(\frac{p_{t_j-t_{j-1}}(y_{\ell,j}-y_{\ell,j-1})}
{(y_{\ell, \pi_\ell(j)}-y_{\ell, \pi_\ell(j-1)})^{H(1+\delta_{\ell,j}\gamma)}}
\bigg)\, d\bar{y}\Bigg].
\end{equation}
Fix $\ell \in \{1,\cdots, d\}$ and a term in $\sum^{''}$, we
proceed to estimate the integral
\begin{equation}\label{estimate-summand-J}
\int_{\RR{R}^n\cap \Delta_{\pi_\ell}}\prod_{j=1}^{n}
\frac{p_{t_j-t_{j-1}}(y_{\ell,j}-y_{\ell,j-1})}{(y_{\ell, \pi_\ell(j)}
-y_{\ell, \pi_\ell(j-1)})^{H(1+\delta_{\ell,j}\gamma)}}\,
d\bar{y_{\ell}},
\end{equation}
where $d\bar{y_{\ell}}= dy_{\ell, 1}\cdots dy_{\ell, n}$.
Note that we can write \eqref{estimate-summand-J} as
\begin{equation}\label{estimate-summand-J2}
\int_{\RR{R}^n\cap \Delta_{\pi_\ell}}\prod_{j=1}^{n}
\frac{p_{t_{\pi_\ell(j)}-t_{\pi_\ell(j)-1}}(y_{\ell,\pi_\ell(j)}
-y_{\ell,\pi_\ell(j)-1})}{(y_{\ell, \pi_\ell(j)}
-y_{\ell, \pi_\ell(j-1)})^{H(1+\delta_{\ell,j}\gamma)}}
d\bar{y_{\ell}}.
\end{equation}
It will be helpful to notice the difference of
$y_{\ell,\pi_\ell(j)-1}$ in the numerator and
$y_{\ell, \pi_\ell(j-1)}$ in the denominator.

Since $p_t(x)=t^{-1/\alpha}p_1(x/t^{1/\alpha})$, for
all $t>0$, $x\in \RR{R}$. Now
\eqref{estimate-summand-J2} can be written as
\begin{equation}\label{estimate-summand-J3}
\int_{\RR{R}^n\cap \Delta_{\pi_\ell}}\prod_{j=1}^{n}
\Bigg(\frac{1}{(t_{\pi_\ell(j)}-t_{\pi_\ell(j)-1})^{1/\alpha}}
\frac{p_{1}\Big(\frac{y_{\ell,\pi_\ell(j)}-y_{\ell,\pi_\ell(j)-1}}
{(t_{\pi_\ell(j)}-t_{\pi_\ell(j)-1})^{1/\alpha}}\Big)}{(y_{\ell, \pi_\ell(j)}
-y_{\ell, \pi_\ell(j-1)})^{H(1+\delta_{\ell,j}\gamma)}}\Bigg)
\, d\bar{y_{\ell}}.
\end{equation}
We now integrate in the order $dy_{\ell,\pi_\ell(n)}, dy_{\ell,\pi_\ell(n-1)},
\cdots, dy_{\ell,\pi_\ell(1)}$.

A change of variables
$$
y_{\ell,\pi_\ell(n)}-y_{\ell,\pi_\ell(n)-1}
=(t_{\pi_\ell(n)}-t_{\pi_\ell(n)-1})^{1/\alpha}\,z_n
$$
gives
\begin{equation}\label{iteration-n}\begin{split}
&\int_{y_{\ell, \pi_\ell(n-1)}}^\infty
\frac{1}{(t_{\pi_\ell(n)}-t_{\pi_\ell(n)-1})^{1/\alpha}}
\frac{p_{1}\Big(\frac{y_{\ell,\pi_\ell(n)}-y_{\ell,\pi_\ell(n)-1}}
{(t_{\pi_\ell(n)}-t_{\pi_\ell(n)-1})^{1/\alpha}}\Big)}
{(y_{\ell, \pi_\ell(n)}-y_{\ell, \pi_\ell(n-1)})^{H(1+\delta_{\ell,j}\gamma)}}
\,dy_{\ell,\pi_\ell(n)}\\
&= \frac{1} {\big(t_{\pi_\ell(n)} - t_{\pi_\ell(n)-1}\big)^{\frac{H(1+\delta_{\ell,j}\gamma)}
{\alpha}}} \int_{s_n}^\infty \frac{p_{1}(z_n)} {(z_n-s_n)^{H(1+\delta_{\ell,j}\gamma)}}
\,dz_n,
\end{split}
\end{equation}
where
$$
 s_n=\frac{y_{\ell, \pi_\ell(n-1)}-y_{\ell, \pi_\ell(n)-1}}
 {(t_{\pi_\ell(n)}-t_{\pi_\ell(n)-1})^{1/\alpha}}.
 $$
Since we have assumed $0 <\gamma  <\frac12(1-H)$, so $H(1+\delta_{\ell,j}\gamma)<1$
for all $\ell, j$. Thus
$$
\int_{s_n}^\infty
\frac{p_{1}(z_n)}{(z_n-s_n)^{H(1+\delta_{\ell,j}\gamma)}
}dz_n \leq K,
$$
where $K$ is a constant independent of $s_n$. This can
be verified directly by splitting the interval $[s_n, \infty)$
into $[s_n, s_n + 1]$ and $[s_n + 1, \infty)$.

Continuing this procedure we derive
\begin{equation}\label{estimate-summand-J4}\begin{split}
& \int_{\RR{R}^n\cap \Delta_{\pi_\ell}}\prod_{j=1}^{n}
\frac{p_{t_{\pi_\ell(j)}-t_{\pi_\ell(j)-1}}(y_{\ell,\pi_\ell(j)}
-y_{\ell,\pi_\ell(j)-1})}{(y_{\ell, \pi_\ell(j)}-
y_{\ell, \pi_\ell(j-1)})^{H(1+\delta_{\ell,j}\gamma)}}
\,d\bar{y_{\ell}}\\
&\leq K^n \prod_{j=1}^n \frac{1} {\big(t_{\pi_\ell(j)}-
t_{\pi_\ell(j)-1}\big)^{\frac{H(1+\delta_{\ell,j}\gamma)}{\alpha}}}.
\end{split}
\end{equation}
Combining this inequality with equation \eqref{J-estimate} gives
\begin{equation}\begin{split}
J&\leq K^n\, (n!)^{\gamma}\sum_{\pi_1,\cdots, \pi_d}
\prod_{\ell=1}^{d}\prod_{j=1}^{n}\frac{1} {\big(t_{\pi_\ell(j)}-t_{\pi_\ell(j)-1}\big)^{\frac{H(1+\delta_{\ell,j}\gamma)}{\alpha}}}\\
&\leq K^n\, (n!)^{\gamma}\sum_{\pi_1,\cdots, \pi_d}\prod_{j=1}^{n}\frac{1}
 {\big(t_{j}-t_{j-1}\big)^{\frac H\alpha(d+\gamma\sum_{\ell=1}^d\delta_{\ell,\pi_\ell^{-1}(j)})}}\\
 &\leq K^n\, (n!)^{\gamma+d}\prod_{j=1}^{n}\frac{1}{(t_{j}-t_{j-1})^{\frac H\alpha(d+\gamma\epsilon_j)}},
\end{split}
\end{equation}
where $0\leq \epsilon_j\leq 2d$ and $\sum_{j=1}^n\epsilon_j\leq 2n$, thanks to (\ref{Eq:delta}).

Hence we have shown
\begin{equation}\begin{split}\label{Eq:diff-final}
\E \big[L(x+w,B)-L(x,B)\big]^{n}&\leq
K^{n}\,|w|^{n\gamma}h^{n(1-(d+\gamma)H/\alpha)}(n!)^{ d+\gamma+1}\\
& \qquad \times \int_{0\leq t_1\leq  \cdots\leq t_n\leq 1}
\frac{d\bar{t}}{\prod_{j=1}^{n}(t_{j}-t_{j-1})^{\frac H\alpha(d+\gamma\epsilon_j)}}\\
&\leq K^{n}\,|w|^{n\gamma}h^{n(1-(d+\gamma)H/\alpha)}(n!)^{ d+\gamma+1}\\
& \qquad \times \frac{\prod_{j=1}^{n}\Gamma(1-\frac H\alpha(d+\gamma\epsilon_j))}
{\Gamma \big(1+n-\sum_{j=1}^n\frac H\alpha(d+\gamma\epsilon_j)\big)}.
\end{split} \end{equation}
In the above we use  Lemma \ref{iterated-time} with the fact that
$ \frac H\alpha(d+\gamma\epsilon_j)<1 $
because  $\gamma $ satisfies $0<\gamma<\frac12(\alpha/(Hd)-1)$.
It is now clear that (\ref{moment-2}) follows from (\ref{Eq:diff-final})
and Stirling's formula. This completes the proof.
\end{proof}

We have similar bounds for the moments of the local time in the case
when $Y$ is a stable subordinator. It should be noted that the
power of $n!$ in (\ref{moments-2-1}) is different from that
in Lemma \ref{moments}. This will lead to different forms
of laws of the iterated logarithm
for the local times in the two cases.
\begin{lemma}\label{moments-2}
Let $X=\{ X(t),\, t\geq 0\}$ be a $d$-dimensional $\alpha$-time
fractional Brownian motion with $d<\alpha/H$ for which $Y(t)$ is a
stable subordiantor with index $\alpha<1$. For any $h>0$, $B=[0,h]$,
$x,w\in \RR{R}^{d}$, any even integer $n\geq 2$ and any
$0<\gamma<\frac{1}{2}\min \{\alpha/(Hd)-1,\,1 -H\}$, we have
\begin{eqnarray}
\E[L(x,B)]^{n}&\leq
&K^{n}h^{(1-dH/\alpha)n}(n!)^{dH/\alpha}\label{moments-2-1},\\
\E[L(x+w,B)-L(x,B)]^{n}&\leq&K^{n}|w|^{n\gamma}
h^{n(1-(d+\gamma)H/\alpha)} (n!)^{\gamma+\frac{H(d+2\gamma)}
{\alpha}},\label{moments-2-2}
\end{eqnarray}
where $K>0$ is a finite constant depending on $d$, $H$, $\gamma$
and $\alpha$ only.
\end{lemma}
\begin{proof}
The proof of (\ref{moments-2-1}) is similar to the
proof of Lemma \ref{moments}. However, since the sample
function $Y(t)$ is increasing, for $t_1, \cdots, t_n$ that
satisfy (\ref{increasing-t}), the corresponding $y_j= Y(t_j)$
($j = 1, \cdots, n$) satisfy $y_1 <\cdots <y_n$, which leads to some clear modifications
to the proof. Equation (\ref{moments-2-2}) follows similarly from the proof of Lemma
\ref{moments-difference}.
\end{proof}

%\begin{proof}
%The proof of (\ref{moments-2-1}) follows closely the
%proof of Lemma \ref{moments}. Since the sample
%function $Y(t)$ is increasing, for $t_1, \cdots, t_n$ that
%satisfy (\ref{increasing-t}), the corresponding $y_j= Y(t_j)$
%($j = 1, \cdots, n$) satisfy $y_1 <\cdots <y_n$. Hence
%$\min\{|y_{j}-y_{i}|^{H}, 0\leq i\leq j-1\}=(y_{j}-y_{j-1})^{H}$
%in (\ref{moment-good-1}) and we do not have $(n!)^{H}$ in
%(\ref{moment-good-1-b}). Instead, we have
%\begin{equation}
%\begin{split}
%\int_{\RR{R}^{n}}\E\exp\bigg(i\sum_{j=1}^{n}u_{j}W(Y(t_{j}))\bigg)\,dU
%&\le K^n \int_{0 < y_1 < \cdots < y_n} \prod_{j=1}^n
%\frac{p_{t_{j}-t_{j-1}} (y_{j}-y_{j-1})} { (y_{j}-y_{j-1})^{H},
%} \, dy_{1} \cdots dy_{n} \nonumber\\
%&\leq K^{n}\prod_{j=1}^{n}\frac{1}{(t_{j}-t_{j-1})^{H/\alpha}}.\nonumber
%\end{split}
%\end{equation}
%This and Lemma \ref{iterated-time} together yield (\ref{moments-2-1}).
%Equation (\ref{moments-2-2}) follows similarly from the proof of Lemma
%\ref{moments-difference}.
%\end{proof}

Now we are ready to prove the joint continuity result of local times.

\begin{theorem}\label{main-thm-1}
If $d<\alpha /H$, then almost surely $X= \{X(t), \, t\geq 0\}$
has a jointly continuous local time $L(x,\,t)$
$(x\in \RR{R}^{d}, t\geq 0)$.
\end{theorem}

\begin{proof}\
The proof follows from Lemmas \ref{moments}, \ref{moments-difference},
\ref{moments-2} and Kolmogorov's continuity theorem.
\end{proof}

\begin{remark}
For $d=1$, $H=1/2$ and $\alpha >1$ this result was proved by Nane
\cite{nane-alpha}. Theorem \ref{main-thm-1} implies that for $d=1$,  $H=1/2$
and $\alpha >1/2$,  almost surely $X(t)$ $(t\geq 0)$ has a jointly
continuous local time $L(x,\,t)$ $(x\in \RR{R}^{d}, \, t\geq 0)$. Hence
Theorem \ref{main-thm-1} is an improvement of the results in
\cite{nane-alpha} and an extension of results in \cite{xiao}
obtained for multidimensional iterated Brownian motion.
\end{remark}

The following tail probability estimates are used in deriving the sharp H\"older
conditions in the set variable of the local times of $\alpha$-time fractional
Brownian motion.

\begin{lemma}\label{tail-prob}
Suppose $Y$ is not a subordinator. For  any $\lambda>0$,
there exists a finite constant $A>0$, depending on $\lambda$,
$d$, $H$ and $\alpha$ only, such that for all $\tau\geq 0$,
$h>0$, $B=[\tau ,\tau +h]$, $x,w\in
\RR{R}^{d}$, all
$0<\gamma<\frac{1}{2}\min \{\alpha/(Hd)-1,\,1 -H\}$, and all $u>0$
\begin{eqnarray}
\P\Big\{L(x+X(\tau),B)\geq Ah^{1-dH/\alpha}u^{dH(1+1/\alpha)}\Big\}
&\leq &\exp(-\lambda u),\label{exponential-1}\\
\P\Big\{|L(x+w+X(\tau),B)-L(x+X(\tau),B)|
&\geq &A|w|^{\gamma}h^{1-(d+\gamma)H/\alpha}u^{C(H,\alpha)}\Big\}\nonumber\\
& \leq & \exp(-\lambda u),\label{exponential-2}
\end{eqnarray}
where $C(H,\alpha)=d+\gamma+\frac{H(d+2\gamma)}{\alpha}$.
\end{lemma}
\begin{proof}\ Since $X=\{X(t), t \ge 0\}$ has stationary increments,
i.e., for any $\tau\ge 0$, the processes $\{X(t+\tau) - X(\tau),
t \ge 0\}$ and $X$ have the same finite dimensional distributions.
Hence Lemmas \ref{moments} and \ref{moments-difference} can be reformulated as follows:
For any $\tau\geq 0$,
$h>0$, $B=[\tau ,\tau +h]$, $x,w\in
\RR{R}^{d}$, any even integer $n\geq 2$ and any
$0<\gamma<\frac{1}{2}\min \{\alpha/(Hd)-1,\,1 -H\}$, we have
\begin{eqnarray}
\E[L(x+X(\tau),B)]^{n}&\leq
&K^{n}h^{(1-dH/\alpha)n}(n!)^{dH(1+1/\alpha)},\label{moment-11}\\
\E[L(x+w+X(\tau),B)-L(x+X(\tau),B)]^{n}
&\leq&K^{n}|w|^{n\gamma}h^{n(1-(d+\gamma)H/\alpha)}\nonumber\\
&& \times  (n!)^{d+\gamma+\frac{H(d+2\gamma)}{\alpha}},\label{moment-21}
\end{eqnarray}
where  $K>0$ is a finite constant depending
on $d$, $H$, $\ga$ and $\alpha$ only.

Now, Lemma \ref{tail-prob} is a direct consequence of (\ref{moment-11}),
(\ref{moment-21}) and the Chebyshev's inequality.
\end{proof}

\begin{lemma}\label{tail-prob-2}
Suppose $Y$ is a stable subordinator of index $\alpha<1$. For any $\lambda>0$,
there exists a finite constant $A>0$, depending on $\lambda$, $d$, $H$
and $\alpha$ only, such that for all
$\tau\geq 0$, $h>0$, $B=[\tau ,\tau +h]$, $x,w\in \RR{R}^{d}$,
all
$0<\gamma<\frac{1}{2}\min \{\alpha/(Hd)-1,\,1 -H\}$, and $u>0$
\begin{eqnarray}
\P\Big\{L(x+X(\tau),B)\geq Ah^{1-dH/\alpha}u^{dH/\alpha}\Big\}
&\leq &\exp(-\lambda u)\label{exponential-2-1},\\
\P\Big\{|L(x+w+X(\tau),B)-L(x+X(\tau),B)|
&\geq &A|w|^{\gamma}h^{1-(d+\gamma)H/\alpha}u^{D(H,\alpha)}\Big\}\nonumber\\
& \leq & \exp(-\lambda u),\label{exponential-2-2}
\end{eqnarray}
where $D(H,\alpha)=\gamma+\frac{H(d+2\gamma)}{\alpha}$.
\end{lemma}
The proof of Lemma \ref{tail-prob-2} follows the same idea as that in the proof of
Lemma \ref{tail-prob}, with an application of Lemma \ref{moments-2}. We omit it here.

The next lemma shows that process the real-valued process $Z(t)=W(Y(t))$ has
heavy tails as in the case of $Y$. This might make this process more desirable,
since it has heavy tails without independence of increments and with the
stationarity of the increments. We need the following lemma to prove Lemma
\ref{cont-prob} for a two sided estimate of $\P\Big\{ \sup_{0\leq t\leq 1}
|Z(t)|> u \Big\},$ which will be useful in proving Theorem \ref{main-thm-2}.

\begin{lemma}\label{asymptotics-x}
Let $d=1$, $0<H<1$ and $0< \alpha \leq 2$, and  let $0\leq a\leq b$ then
$$
\lim_{u\to\infty}\frac{\P\Big\{|Z(b)-Z(a)|>u\Big\}}{u^{-\alpha/H}}= C(b-a)
$$
for some finite constant $C>0$.
\end{lemma}

\begin{proof}\,
By using the stationarity of the increments and the self-similarity of
$W$ and $Y$ we  get
\begin{equation}\label{ldp-prob}\begin{split}
\P\Big\{|Z(b)-Z(a)|>u\Big\}&= \P\Big\{|W(Y(b-a))|>u\Big\}\\
&=\P\Big\{(b-a)^{H/\alpha}|Y(1)|^{H}|W(1)|>u\Big\}\\
&=\int_{-\infty}^\infty \P\Big\{(b-a)^{H/\alpha}|Y(1)|^{H}|s|>u\Big\}f^{H}(s)ds\\
%&=\int_{-\infty}^\infty \P\Big\{|Y(1)|^{H}>u(b-a)^{-H/\alpha}|s|^{-1}\Big\}p^{H}(s)ds\\
&=\int_{-\infty}^\infty \P\Big\{|Y(1)|>u^{1/H}(b-a)^{-1/\alpha}
|s|^{-1/H}\Big\}f^{H}(s)ds,
\end{split}
\end{equation}
here $f^H(s)=\frac{e^{-s^2/2}}{\sqrt{2\pi}}$ is the density of $W(1)$.

The following is a well-known result
$$
\lim_{u\to\infty}\frac{\P\Big\{|Y(1)|>u\Big\}}{u^{-\alpha}}= k
$$
for some $k>0$; see, for example, Bertoin \cite{bertoin}.
Hence, for fixed $a\leq b$, $s\in \RR{R}$, and as $x\to\infty$
\begin{equation}
\begin{split}
P\Big\{|Y(1)|>x^{1/H}(b-a)^{-1/\alpha}|s|^{-1/H}\Big\}&\sim
k (x^{1/H}(b-a)^{-1/\alpha}|s|^{-1/H})^{-\alpha}\\
&=k x^{-\alpha/H}(b-a)|s|^{\alpha/H}.
\end{split}
\end{equation}

Now we apply the Dominated Convergence Theorem in equation
\eqref{ldp-prob} to get
\begin{equation}\begin{split}
\lim_{x\to\infty}\frac{\P\Big\{|Z(b)-Z(a)|>x\Big\}}{x^{-\alpha/H}}&=
k(b-a)\int_{-\infty}^\infty|s|^{\alpha/H}f^H(s)ds\\
&=k(b-a)2^{\alpha/2H}\pi^{-1/2}\Gamma((\alpha+H)/2H).
\end{split}\end{equation}
Hence the constant in the theorem is $C=k2^{\alpha/2H}\pi^{-1/2}\Gamma((\alpha+H)/2H)$.
%\alert{So we can use the fact that $Y$} is symmetric. I guess a similar
\end{proof}

\begin{lemma}\label{cont-prob}
Let $d=1$, $0<H<1$ and $0< \alpha \leq 2$. There exists a
finite constant $K>0$ such that for $u \ge1$,
\begin{equation}\label{Eq:Tail}
K^{-1} u^{-\alpha/H}\leq \P\Big\{ \sup_{0\leq t\leq 1} |Z(t)|> u \Big\}
\leq K u^{-\alpha/H}.
\end{equation}
\end{lemma}

\begin{proof}\ Let $S(t)\equiv \sup_{0\leq s \leq t} |Y(s)|$.
Then, by using the scaling property of $W$ and conditioning, we have
\begin{equation}\label{Eq:Tail2}
\begin{split}
\P\bigg\{ \sup_{0\leq t\leq 1} |Z(t)|> u \bigg\}
&\leq \P\bigg\{ \sup_{|x|\leq S(1)} |W(x)|> u \bigg\}\\
&= \E\Bigg(\P\bigg\{\sup_{|x|\leq 1} |W(x)|> \frac{u}
{S(1)^H} \Big| Y\bigg\}\Bigg).
\end{split}
\end{equation}
It is well known that, for any $\ep > 0$, there exists
a finite constant $K$ such that for all $u>0$
\begin{equation}\label{Eq:maxFBm}
\P\bigg\{\sup_{|x| \leq 1}|W(x)|>u\bigg\}
\leq K\, \exp\Big(-\frac{u^{2}} {2+\ep}\Big).
\end{equation}
See, for example, Lifshits \cite[Section 14]{lifshits}.
Consequently
\begin{equation}\label{Eq:Tail3}
\P\Big\{\sup_{0\leq t\leq 1} |Z(t)|> u \Big\}
\leq K\, \E \exp\bigg(-\frac{u^{2}} {(2+\ep) S(1)^{2H}}\bigg).
\end{equation}
Since, for all $x > 0$, the function $g(x) =
\exp\big(-\frac{u^{2}}{(2+\ep)x^{2H}}\big)$
has positive derivative
$$
g'(x) = \frac{2H\, u^2}{2+\ep}\,
\exp\bigg(-\frac{u^{2}}{(2+\ep)x^{2H}}\bigg) \frac1 {x^{2H+1}},$$
we derive
\begin{equation}\label{Eq:Tail4}
\begin{split}
&\P\bigg\{\sup_{0\leq t\leq 1} |Z(t)|> u \bigg\}\\
&\leq K u^2\,\int_0^\infty  \exp\bigg(-\frac{u^{2}}{(2+\ep)x^{2H}}\bigg)
\frac1 {x^{2H+1}} \P\Big\{S(1) > x\Big\}\, dx\\
&= K\, \int_0^\infty \exp\bigg(-\frac{1}{(2+\ep)y^{2H}}\bigg)
\frac1 {y^{2H+1}} \P\Big\{S(1) > u^{1/H}y\Big\}\, dy,
\end{split}
\end{equation}
where the last inequality follows from the change of variable
$x= u^{1/H}y$. Now by using the well-known estimate
\[
\P\big\{S(1) >  y\big\} \le K \,\big(1 \wedge y^{-\alpha}\big),
\qquad \forall \ y > 0,
\]
we obtain that for all $u > 1$,
\[
\int_0^\infty \exp\bigg(-\frac{1}{(2+\ep)y^{2H}}\bigg)
\frac1 {y^{2H+1}} \P\Big\{S(1) > u^{1/H}y\Big\}\, dy
\le K \, u^{- \alpha/H},
\]
where $K > 0$ is a finite constant. This and (\ref{Eq:Tail4})
together give the upper bound in (\ref{Eq:Tail}).

The lower bound in \eqref{Eq:Tail} follows from Lemma
\ref{asymptotics-x} and the fact that
\begin{equation}\label{upper-lower-bound}
\P\big\{|Z(1)|> u \big\}=
\sup_{0\leq t\leq 1}\P\big\{|Z(t)|> u \big\}
\leq \P\bigg\{ \sup_{0\leq t\leq 1} |Z(t)|> u \bigg\}.
\end{equation}
The first equality in Equation \eqref{upper-lower-bound}
follows from the fact that the function
$$t\to\P\big\{|Z(t)|> u \big\}=\P\Big\{|Z(1)|> t^{-H/\alpha}u \Big\}
$$ is an increasing function for $t\in (0,1]$.
\end{proof}

The following theorems are for laws of the iterated logarithm
for the maximum local time $L^*([\tau, \tau +h]) =
\sup_{x\in \RR{R}^{d}}L(x,[\tau, \tau +h])$ and uniform H\"older conditions of local
times of $\alpha$-time fractional Brownian motions.

\begin{theorem}\label{main-thm-2}
Let $d<\alpha/H$ and suppose $Y$ is not a subordinator.
\begin{description}
\item (1). There exists a finite
constant $K>0$ such that for any $\tau\geq 0$ with probability $1$
\begin{equation}\label{holder-1}
\limsup_{h\to 0}\sup_{x\in \RR{R}^{d}}\frac{L(x,\tau
+h)-L(x,\tau)}{h^{1-dH/\alpha}(\log \log
h^{-1})^{dH(1+1/\alpha)}}\leq K.
\end{equation}
\item (2). For any $T>0$, there exists a positive constant $K$
such that almost surely
\begin{equation}\label{holder-2}
\limsup_{h\to 0}\sup_{0\leq t\leq T}\sup_{x\in
\RR{R}^{d}}\frac{L(x,t+h)-L(x,t)}{h^{1-dH/\alpha}(\log 1/h)^{
dH(1+1/\alpha ) }}\leq K.
\end{equation}
\end{description}
\end{theorem}
\begin{proof}\ Eq. (\ref{holder-1}) follows from Lemma \ref{cont-prob} and a chaining argument as that in the proof
of Theorem 2 in \cite{xiao}. The proof of Eq. (\ref{holder-2}), using Lemma \ref{tail-prob}, is very similar to that of
Xiao \cite[Theorem 3]{xiao} and Ehm \cite[Theorem 2.1]{ehm}.
We omit the details.
\end{proof}

\begin{theorem}\label{main-thm-2-2}
Let $d<\alpha / H$ and suppose $Y$ is a stable subordinator of
index $\alpha<1$.
\begin{description}
\item (1). There exists a finite constant $K>0$ such
that for any $\tau\geq 0$ with probability $1$
\begin{equation}\label{holder-2-1}
\limsup_{h\to 0}\sup_{x\in \RR{R}^{d}}\frac{L(x,\tau
+h)-L(x,\tau)}{h^{1-dH/\alpha}(\log \log
h^{-1})^{dH/\alpha}}\leq K.
\end{equation}
\item (2). For any $T>0$,
there exists a finite
constant $K>0$ such that almost surely
\begin{equation}\label{holder-2-2}
\limsup_{h\to 0}\sup_{0\leq t\leq T}\sup_{x\in
\RR{R}^{d}}\frac{L(x,t +h)-L(x,t)}{h^{1-dH/\alpha}(\log
h^{-1})^{dH/\alpha}}\leq K.
\end{equation}
\end{description}
\end{theorem}

%\begin{proof}
%Eq. (\ref{holder-2-1}) follows the same lines as in the proof of
%Theorem 2 in  \cite{xiao} by using Lemma \ref{tail-prob-2}.
%\end{proof}

The H\"{o}lder conditions for the local time of a stochastic process
$X(t)$ are closely related to the irregularity of the sample paths
of $X(t)$ (cf. Berman \cite{Berman69}). In the following, we will
apply Theorems \ref{main-thm-2} and \ref{main-thm-2-2} to derive results about
the degree of oscillation of the sample paths of $X(t)$.

\begin{theorem}\label{main-thm-4}
Suppose $Y$ is not a stable subordinator. Let $X=
\{X(t),\,t\in \RR{R}_{+}\}$ be an $\alpha$-time fractional
Brownian motion in $\R^d$ with $H < \alpha$. For any $\tau\in
\RR{R}_{+}$, there exists a finite constant $K>0$ such that
\begin{equation}\label{osc-bound-1}
\liminf_{r\to 0}\sup_{s\in B(\tau ,r)}
\frac{|X(s)-X(\tau)|}{r^{H/\alpha}/(\log \log
1/r)^{H(1+1/\alpha)}}\geq K \ \ \ \  a.s.
\end{equation}
For any interval $T\subset \RR{R}_{+} $
\begin{equation}\label{osc-bound-2}
\liminf_{r\to 0}\inf_{t\in T}\sup_{s\in B(t ,r)}
\frac{|X(s)-X(\tau)|}{r^{H/\alpha}/(\log 1/r)^{H(1+1/\alpha)}}
\geq K \ \ \ \ a.s.
\end{equation}
In particular, $X(t)$ is almost surely nowhere differentiable
in $\RR{R}_{+}$.
\end{theorem}
\begin{proof}
Clearly, it is sufficient to consider the case of $d=1$,
where the condition of Theorem \ref{main-thm-2}
(i.e. $1<\alpha/H$) is fulfilled.
For any interval $Q\subset \RR{R}_{+}$,
\begin{eqnarray}\label{sup-bound}
\lambda_{1}(Q)=\int_{\overline{X(Q)}}L(x,Q)dx
\leq  L^{*}(Q) \bigg( \sup_{s,t\in Q}|X(s)-X(t)|\bigg).
\end{eqnarray}
Let $Q=B(\tau,r)$. Then (\ref{osc-bound-1}) follows
immediately from (\ref{holder-1}) and (\ref{sup-bound}).
Similarly (\ref{osc-bound-2}) follows from (\ref{holder-2})
and (\ref{sup-bound}).
\end{proof}

\begin{remark}
Theorem \ref{main-thm-4} extends partially the results
obtained by Nane \cite{nane-alpha}.
\end{remark}

\begin{theorem}\label{main-thm-2-4}
Suppose $Y$ is a stable subordinator of index $\alpha<1$.
Let $X= \{X(t), t\in \RR{R}_{+}\}$ be $\alpha$-time fractional
Brownian motion in $\R^d$ with $H<\alpha$. For any $\tau\in \RR{R}_{+}$,
there exists a finite constant $K>0$ such that
\begin{equation}\label{osc-bound-2-1}
\liminf_{r\to 0}\sup_{s\in B(\tau ,r)}
\frac{|X(s)-X(\tau)|}{r^{H/\alpha}/(\log \log
1/r)^{H/\alpha}}\geq K \ \ \ \  a.s.
\end{equation}
For any interval $T\subset \RR{R}_{+} $
\begin{equation}\label{osc-bound-2-2}
\liminf_{r\to 0}\inf_{t\in T}\sup_{s\in B(t ,r)}
\frac{|X(s)-X(\tau)|}{r^{H/\alpha}/(\log
1/r)^{H/\alpha}}\geq K \ \ \ \  a.s.
\end{equation}
In particular, $X(t)$ is almost surely nowhere differentiable
in $\RR{R}_{+}$.
\end{theorem}

\noindent{\bf Acknowledgments.}  Authors  would like to thank the two referees for their  comments and corrections  that helped improve the paper.

\section{Appendix}\label{sec:Appendix}

As an appendix, we provide the following lemmas,  which are used in the
proofs of our main results in Section \ref{sec:continuity}.
Lemma \ref{XIAO-lemma} is from Xiao \cite{xiao}, which is
used to prove Lemma \ref{mainlemma}.

\begin{lemma}\label{XIAO-lemma}
Let $0< \gamma <1$ be a constant. Then for any integer $n\geq 1$ and
any $x_{1},\cdots ,x_{n}\in \RR{R}$, we have
$$
\int_{0}^{1}\frac{1}{\min \{|x-x_{j}|^{\gamma}, j=1,\cdots,n\}} \,
dx \leq K\, n^{\gamma},
$$
where $K>0$ is a finite constant depending only on $\gamma$.
\end{lemma}

\begin{lemma}\label{mainlemma}
Let $0< \gamma <1$ be a constant. Then for any integer $n\geq 1$ and
any $x_{1}, \cdots ,x_{n}\in \RR{R}$, we have
\begin{equation}\label{mainlemma-eqn}
\int_{\RR{R}}\frac{p_{1}(x)} {\min \{|x-x_{j}|^{\gamma},
j=1,\cdots,n\}}\, dx\leq K\, n^{\gamma},
\end{equation}
where $K>0$ is a finite constant depending only on $\gamma$ and
$\alpha$.
\end{lemma}
\begin{proof}\
We recall the following asymptotic bounds from  \cite{skorokhod} for
the stable density function  $p_{1}(x)$ as $x\to\infty$ (the
asymptotics for the case $x \to -\infty$ are obtained by changing
$x$ to $-x$). For $0 < \alpha <1$:
$$
p_{1}(x)\leq K\, x^{-(1+\alpha)}, \ \ \hbox{ as } \ \ x\to\infty.
$$
For $\alpha=1$ and $\beta = 0$ (this is the symmetric Cauchy case):
$$
p _{1}(x)\leq K\,  x^{-2}\ \ as \ \ x\to\infty .
$$
For $\alpha >1$ and $ -1<\beta <1$:
$$
p_{1}(x)\leq Kx^{-(1+\alpha)}, \ \ \hbox{ as } \ \ x\to\infty .
$$
For $\alpha >1$ and $\beta=-1, 1$:
$$
p_{1}(x)\leq K\, \max  \Big\{ x^{-(1+\alpha)},\,
x^{-1+\alpha/2(\alpha -1)}
\exp\big(-c(\alpha)x^{\alpha/(\alpha-1)}\big) \Big\},
\ \ \hbox{ as } \ \ x \to \infty.
$$

%$$
%For $\alpha=1, \beta=1$:
%$$
%p^{\alpha}_{1}(x+\beta\frac{2}{\pi }\ln x)\leq K x^{-2}\ \ as \ \ x\to\infty .
%$$
%For $\alpha=1, \beta=1$:
%$$
%p^{\alpha}_{1}(x)\leq K\exp(-\frac{\pi}{4}(-x))\ \ as \ \ x\to -\infty .
%$$
%For $\alpha=1, \beta=-1$:
%$$
%p^{\alpha}_{1}(x)\leq K\exp(-\frac{\pi}{4}x)\ \ as \ \ x\to \infty .
%$$
%For
%$\alpha=1, \beta=-1$:
%$$
%p^{\alpha}_{1}(x+\beta\frac{2}{\pi }\ln x)\leq K x^{-2}\ \ as \ \ x\to -\infty .
%$$
Now we observe that the left-hand side of (\ref{mainlemma-eqn})
can be written as
\begin{equation}\label{L1}
\begin{split}
&\sum_{l\in \RR{Z}} \int_{l}^{l+1} \frac{p_{1}(x)}
{\min \{|x-x_{j}|^{\gamma}, j=1,\cdots,n\}}\, dx \\
&\leq \max_{|x|\le M} p_{1}(x)\, \int_{-M}^{M}
\frac{1} {\min \{|x-x_{j}|^{\gamma}, j=1,\cdots,n\}}\, dx \\
& \qquad + \sum_{|l|>M} \max_{l\leq x \leq l+1 } p_{1}(x)\,
\int_{0}^{1}\frac{1}{\min \{|x+l-x_{j}|^{\gamma}, j=1,\cdots,n\}}\,
dx.
\end{split}
\end{equation}
It can be verified that Equation (\ref{mainlemma-eqn}) follows from
(\ref{L1}), the asymptotics of $p_{1}(x)$ and Lemma
\ref{XIAO-lemma}.
\end{proof}

Lemma \ref{iterated-time} is taken from Ehm \cite{ehm} and
Lemma \ref{Lem:Cuzick} is due to Cuzick and DuPreez
\cite{CD82} (the current form is from Khoshnevisan and Xiao
\cite{KX04}).

\begin{lemma}\label{iterated-time}
For any integer $n\geq 1$, and $\beta_j\in (0,1)$ for
$1\leq j\leq n$, for all $h>0$, we have
\[
\begin{split}
&\int_{0\leq x_1\leq x_2\leq \cdots\leq x_n\leq h}
\prod_{j=1}^n\frac{1}{(x_j-x_{j-1})^{\beta_j}}dx_1\cdots dx_n\\
&\ \qquad =h^{n-\sum_{j=1}^n\beta_j}\frac{\prod_{j=1}^n
\Gamma (1-\beta_j)}{\Gamma(1+n-\sum_{j=1}^n\beta_j)}.
\end{split}
\]
\end{lemma}

\begin{lemma}
\label{Lem:Cuzick} Let $\xi_1, \ldots, \xi_n$ be mean zero Gaussian
variables which are linearly independent, then for any nonnegative
function $g:\,\R\rightarrow\R_+$,
\begin{equation}\label{Eq:Cuzick}
\begin{split}
&\int_{{\R}^n} g(v_1) \exp \biggl[- \frac1 2 {\rm
Var}\bigg(\sum_{j=1}^n v_j \xi_j\bigg)
\biggr]\, dv_1 \cdots dv_n \\
&\ \qquad = \frac{(2 \pi)^{(n-1)/2}} {({\rm detCov}(\xi_1, \cdots,
\xi_n))^{1/2}}\ \int_{- \infty}^{\infty} g\Bigl(\frac {v} {
\sigma_1}\Bigr)\, e^{- v^2/2}\, dv ,
\end{split}
\end{equation}
where ${\rm detCov}(\xi_1, \cdots, \xi_n)$ denotes the determinant
of the covariance matrix of the Gaussian random vector $(\xi_1,
\ldots, \xi_n)$, and where $\sigma_1^2 = {\rm Var}(\xi_1|\xi_2,
\ldots, \xi_n)$ is the conditional variance of $\xi_1$ given
$\xi_2,$ $ \ldots, \xi_n$.
\end{lemma}

\bigskip

\begin{quote}
\begin{small}

 \noindent \textsc{Erkan Nane}.\
Department of Mathematics and Statistics\\
Auburn University\\
221 Parker Hall, Auburn, AL 36849\\
E-mail: \texttt{nane@stt.msu.edu}\\
URL: \texttt{www.duc.auburn.edu/\~{}ezn0001/}\\

 \noindent \textsc{Dongsheng Wu}.\
        Department of Mathematical Sciences, 201J Shelby Center,
        University of Alabama in Huntsville,
        Huntsville, AL 35899, U.S.A.\\
        E-mail: \texttt{dongsheng.wu@uah.edu}\\
        URL: \texttt{http://webpages.uah.edu/\~{}dw0001}\\

  \noindent \textsc{Yimin Xiao}.\
        Department of Statistics and Probability, A-413 Wells
        Hall, Michigan State University,
        East Lansing, MI 48824, U.S.A.\\
        E-mail: \texttt{ xiao@stt.msu.edu}\\
        URL: \texttt{http://www.stt.msu.edu/\~{}xiaoyimi}
\end{small}
\end{quote}


\begin{thebibliography}{99}

\bibitem{adler}
R. J. Adler, {\it The Geometry of Random Fields}. Wiley, New York,
1981.

\bibitem{allouba1} H. Allouba and W. Zheng, \emph{Brownian-time
processes: the pde connection and the half-derivative generator,}
Ann. Probab. {\bf 29} (2001), 1780--1795.

\bibitem{aurzada-lifshits}F. Aurzada, M. Lifshits,
\emph{On the Small deviation problem for some iterated processes},
Electron. J. Probab. {\bf 14} (2009), 1992--2010.

\bibitem{bmn-07}
B. Baeumer, M. M. Meerschaert and E. Nane,
Brownian subordinators and fractional Cauchy problems.
\emph{Trans. Amer. Math. Soc.} {\bf 361} (2009), 3915--3930.

\bibitem{bmn-09}
B. Baeumer, M. M. Meerschaert and E. Nane,
Space-time duality for fractional diffusion.
\emph{J. Appl. Probab..} {\bf 46} (2009), 1100--1115.

\bibitem{beghin-et-al} L. Beghin, L. Sakhno and E. Orsingher.
\emph{Equations of Mathematical Physics and composition of Brownian
and Cauchy processes}. Arxiv id:1008.0928v1.

\bibitem{Berman69}
S. M. Berman, \emph{Local times and sample function properties of
stationary Gaussian processes}, Trans. Amer. Math. Soc. {\bf 137}
(1969), 277--299.

\bibitem{Berman73}
S. M. Berman, \emph{Local nondeterminism and local times of Gaussian
processes,} Indiana Univ. Math. J. {\bf 23} (1973), 69--94.

\bibitem{bertoin} J. Bertoin, \emph{L\'{e}vy Processes,} Cambridge
University Press, 1996.


\bibitem{burdzy1} K. Burdzy, \emph{Some path properties of iterated Brownian
motion,} In Seminar on Stochastic Processes (E. \c{C}inlar, K.L.
Chung and M.J. Sharpe, eds.), pp. 67--87, Birkh\"{a}user, Boston,
1993.

\bibitem{burdzy-khos} K. Burdzy and D. Khoshnevisan, \emph{The level
set of iterated Brownian motion,} S\'{e}minaire de Probabilit\'{e}s
XXIX (Eds.: J Az\'{e}ma, M. Emery, P.-A. Meyer and M. Yor), Lecture
Notes in Mathematics, {\bf 1613}, pp. 231--236, Springer, Berlin,
1995.

\bibitem{BK98}
K. Burdzy and D. Khoshnevisan, \emph{Brownian motion in a Brownian crack,
} Ann. Appl. Probab. {\bf 8} (1998), 708--748.

%\bibitem{debruijn} N. G. De Bruijn, \emph{Asymptotic methods in analysis},
%North-Holland Publishing Co., Amsterdam, 1957.


\bibitem{CCFR} E. Cs\'{a}ki, M. Cs\"{o}rg\"{o}, A. F\"{o}ldes, and
P. R\'{e}v\'{e}sz, \emph{The local time of iterated Brownian
motion,} J. Theoret. Probab. {\bf 9} (1996), 717--743.

\bibitem{CD82}
J. Cuzick and J. DuPreez, \emph{Joint continuity of Gaussian local
times}, Ann. Probab. {\bf 10} (1982), 810--817.

\bibitem{davydov}
Y. Davydov, {\it The invariance principle for stationary processes,}  Teor.
Verojatnost. i Primenen. {\bf 15}  (1970), 498--509.

\bibitem{deblassie2} R. D. DeBlassie, \emph{Higher order PDE's
and symmetric stable processes,} Probab. Theory Relat. Fields {\bf
129} (2004), 495--536.

\bibitem{deblassie} R. D. DeBlassie, \emph{Iterated Brownian
motion in an open set,} Ann. Appl. Probab. {\bf 14} (2004),
1529--1558.

\bibitem{dovidio-orsingher}
M. D'Ovidio and E. Orsingher. \emph{Composition of processes
and related partial differential equations}, J. Theor. Probab.
Online First, April 21, 2010. Doi:10.1007/s10959-010-0284-9.

\bibitem{ehm}
W. Ehm, \emph{Sample function properties of multi-parameter
stable processes}, Z. Wahrsch. verw. Geb. {\bf 56} (1981), 195--228.

\bibitem{EM02}
P. Embrechts and M. Maejima, \emph{Selfsimilar Processes}, Princeton
University Press, Princeton, 2002.

\bibitem{GH80}
D. Geman and J. Horowitz, \emph{Occupation densities}, Ann. Probab.
{\bf 8} (1980), 1--67.

\bibitem{hahn-kobayashi-umarov}
M. Hahn, K. Kobayashi and S. Umarov. Fokker-Plank-Kolmogorv equations
associated with SDEs driven by time-changed fractional Brownian motion.
(2010) arxirv id: 1002.1494v1.

\bibitem{Hu99}
Y. Hu, \emph{Hausdorff and packing measures of the level sets of
iterated Brownian motion}, J. Theoret. Probab. {\bf 12} (1999),
313--346.

\bibitem{kahane} J. P. Kahane, \emph{Some Random Series of Functions,}
second edition, Cambridge University Press, 1985.

\bibitem{KX04}
D. Khoshnevisan and Y. Xiao, \emph{Images of the Brownian sheet},
Trans. Amer. Math. Soc. {\bf 359} (2007), 3125--3151.

\bibitem{lifshits} M. A. Lifshits, \emph{Gaussian Random Functions,}
Kluwer Academic Publishers, Dordrecht,  1995.


\bibitem{linde-shi} W. Linde and Z. Shi. \emph{Evaluating the
small deviation probabilities for subordinated Lévy processes.}
 Stoch. Process. Appl. {\bf 113} (2004), 273--287.

\bibitem{nane1} E. Nane, \emph{Iterated Brownian motion
in parabola-shaped domains,}
Potential Analysis, {\bf 24} (2006), 105--123.

\bibitem{nane2} E. Nane, \emph{Iterated Brownian motion
in bounded domains in $\RR{R}^{n}$,} Stoch. Process. Appl.
{\bf 116} (2006), 905--916.

\bibitem{nane-alpha}
E. Nane, \emph{Laws of the iterated logarithm for $\alpha$-time
Brownian motion}, Electron. J. Probab. {\bf 11}
(2006), 434--459.

\bibitem{nane3} E. Nane, \emph{Higher order PDE's and
iterated processes}, Trans. Amer. Math. Soc. {\bf 360}
(2008), 2681--2692.

\bibitem{nane-lil}
E. Nane, \emph{ Laws of the iterated logarithm for
a class of iterated processes}, Statist. Probab. Letters.
{\bf 79} (2009), 1744--1751.

\bibitem{OB1}
E. Orsingher and L. Beghin, Fractional diffusion equations
and processes with randomly varying time, {\it Ann. Probab.}
{\bf 37} (2009), 206--249.

\bibitem{Pitt78}
L. D. Pitt, \emph{Local times for Gaussian vector fields,}
Indiana Univ. Math. J. {\bf 27} (1978), 309--330.

\bibitem{ST94}
G. Samorodnitsky and M. S. Taqqu,  {\it Stable non-Gaussian Random
Processes:  Stochastic models with infinite variance.} Chapman \&
Hall, New York, 1994.

\bibitem{sato}
K.I. Sato,
\newblock {\it L\'evy Processes and Infinitely Divisible Distributions}.
\newblock Cambridge University Press, 1999.

%\bibitem{spitzer} F. Spitzer, \emph{Principles of random walk,}
%Van Nostrand, Princeton, N.J. (1964).

\bibitem{skorokhod} A. V. Skorokhod, \emph{Asymptotic formulas
for stable distribution
laws,} Selected Translations in Mathematical Statistics and
Probability, {\bf 1} (1961), 157--162; Dokl. Akad. Nauk. SSSR,
{\bf 98} (1954), 731--734.

\bibitem{taqqu}
M. S. Taqqu, \emph{Weak Convergence to fractional Brownian
motion and to the Rosenblatt process,} Z. Wahrsch. Verw.
Gebiete {\bf 31} (1975), 287--302.

\bibitem{Taylor67}
S. J. Taylor, \emph{Sample path properties of a transient stable
process,}  J. Math. Mech. {\bf 16} (1967), 1229--1246.

\bibitem{whitt}
W. Whitt, {\it Stochastic-Process Limits,} Springer, New York,
2002.

\bibitem{xiao2} Y. Xiao, \emph{H\"{o}lder conditions for the local times and
Hausdorff measure of the level sets of Gaussian random fields,}
Probab. Theory Relat. Fields {\bf 109} (1997), 129--157.

\bibitem{xiao} Y. Xiao, \emph{Local times and related properties of
multi-dimensional iterated Brownian motion,} J. Theoret. Probab.
{\bf 11} (1998), 383--408.

\end{thebibliography}
\end{document}